\documentclass[12pt]{amsart}
\usepackage[cp1251]{inputenc}
\usepackage[T2A]{fontenc}
\usepackage[english]{babel}
\usepackage{amsmath,amsfonts,amssymb}
\usepackage{geometry}
\usepackage{amsmath, amsthm, amscd, amsfonts, amssymb, graphicx, color}
\usepackage[bookmarksnumbered, colorlinks, plainpages]{hyperref}

\numberwithin{equation}{section}
\newtheorem{theorem}{Theorem}[section]
\newtheorem{lemma}[theorem]{Lemma}
\newtheorem{proposition}[theorem]{Proposition}

\theoremstyle{remark}
\newtheorem{remark}[theorem]{Remark}

\theoremstyle{definition}
\newtheorem{example}[theorem]{Example}

\begin{document}

\title[Elliptic problems in H\"ormander--Roitberg spaces]{Elliptic problems with boundary operators \\ of higher orders in H\"ormander--Roitberg spaces}


\author[T.~Kasirenko]{Tetiana Kasirenko}
\address{Institute of Mathematics, National Academy of Sciences of Ukraine, Tere\-shchen\-kivska Str. 3, 01004 Kyiv-4, Ukraine}
\email{kasirenko@imath.kiev.ua}


\author[A. Murach]{Aleksandr Murach}
\address{Institute of Mathematics, National Academy of Sciences of Ukraine, Tere\-shchen\-kivska Str. 3, 01004 Kyiv-4, Ukraine}
\email{murach@imath.kiev.ua}




\begin{abstract}
We investigate elliptic boundary-value problems for which the maximum of the orders of the boundary operators is equal to or greater than the order of the elliptic differential equation. We prove that the operator corresponding to an arbitrary problem of this kind is bounded and Fredholm between appropriate Hilbert spaces which form certain two-sided scales and are built on the base of isotropic H\"ormander spaces. The differentiation order for these spaces is given by an arbitrary real number and positive function which varies slowly at infinity in the sense of Karamata. We establish a local \textit{a priori} estimate for the generalized solutions to the problem and investigate their local regularity (up to the boundary) on these scales. As an application, we find sufficient conditions under which the solutions have continuous classical derivatives of a given order.
\end{abstract}

\maketitle

\section{Introduction}\label{sec1}

Among elliptic boundary-value problems occur problems with boundary conditions whose orders are equal to or greater than the order of the elliptic differential equation for which the problem is posed. Such problems appear specifically in acoustic, hydrodynamics, and the theory of stochastic processes \cite{Krasil'nikov61, Venttsel59, VeshevKouzov77}. The known Ventcel' elliptic boundary-value problem \cite{Venttsel59} apparently was the first such a problem arisen in applications. It consists of an elliptic differential equation of the second order and a boundary condition of the same order and was applied initially to investigation of diffusion processes. A number of papers are devoted to this problem; see, e.g. \cite{Bonnaillie-NoelDambrineHerauVial10, LukyanovNazarov98, LuoTrudinger91}. In acoustic, of an interest is the elliptic problem that consists of the Helmholtz equation and a certain boundary condition of the fifth order \cite{Krasil'nikov61, VeshevKouzov77}. From the theoretical point of view, the simplest example of such elliptic problems is the problem consisting of the Laplace equation and the boundary condition $\partial^{k}u/\partial\nu^{k}=g$, where $k\geq2$ and $\nu$ is a unit vector of the normal to the boundary of the domain in which the equation is given \cite{Bitsadze90, Sokolovskiy88} (some generalizations of this example are considered in \cite{Karachik92, Karachik96}). All these elliptic problems are not regular so that the classical Green formula does not hold for them, which complicates their investigation.

The basic properties of general elliptic boundary-value problems consist in that these problems are Fredholm on appropriate pairs of H\"older or positive Sobolev spaces and that their solutions admit \emph{a priory} estimates and satisfy the property of increase in smoothness in these spaces; see, e.g., the fundamental paper by Agmon, Douglis, and Nirenberg \cite{AgmonDouglisNirenberg59} and Agranovich's survey \cite{Agranovich97}. These properties relate in particular to the nonregular elliptic problems mentioned above.

Of great interest in applications are elliptic problems whose right-hand sides are irregular distributions. Such problems appear specifically in investigations of Green function of elliptic problems, in study of elliptic problems with power singularities in the right-hand sides, and in the spectral theory of elliptic differential operators; see, e.g., monographs \cite{Berezansky68, Roitberg96}. The corresponding theory of elliptic problems in spaces of distributions was developed by Berezansky, Krein, and Roitberg \cite{BerezanskyKreinRoitberg63, Berezansky68, Roitberg64, Roitberg65, Roitberg96}, H\"ormander \cite{Hermander63}, Lions and Magenes \cite{LionsMagenes72}, and Schechter \cite{Schechter60, Schechter63a, Schechter63b, Schechter64}. Its main achievements are theorems on complete collections of isomorphisms realized by elliptic problems on two-sided scales of normed spaces constructed on the base of Sobolev spaces. Briefly saying, these theorems assert that the operator generated by an elliptic problem sets an isomorphism between appropriate distribution spaces whose regularity indexes equal $s$ and $s-2q$ respectively, where $s$ is an arbitrary real number and $2q$ is the even order of the corresponding elliptic differential equation. Such theorems were proved for regular elliptic problems, the classical Green formula playing a decisive role in proofs.

Apparently, the isomorphism theorem by Roitberg \cite{Roitberg64, Roitberg65} is the most meaningful among these achievements at least because other results can be deduced \cite{Roitberg68} from this theorem. Later on Roitberg and Kostarchuk \cite{Roitberg69, Roitberg70, KostarchukRoitberg73} extended this isomorphism theorem over nonregular elliptic problems, specifically \cite{KostarchukRoitberg73} over problems with boundary conditions of higher orders with respect to the order of the corresponding elliptic differential equation. This theorem and its various applications are set forth in monographs by Berezansky \cite{Berezansky68} (for regular elliptic problems), Roitberg \cite{Roitberg96, Roitberg99}, Kozlov, Maz'ya and Rossmann \cite{KozlovMazyaRossmann97}, and Agranovich's survey \cite{Agranovich97}. As a rule, the isomorphism theorem is formulated separately for elliptic problems with low orders boundary conditions  and with higher orders ones;
see \cite[Theorems 4.1.2 and 4.1.3]{Roitberg96}  (or
\cite[Theorems 3.4.1 and 4.1.4]{KozlovMazyaRossmann97} stated in terms of the Fredholm property). To describe the range of the operator generated by a nonregular elliptic problem, Kozlov, Maz'ya and Rossmann use a special Green formula and  corresponding elliptic boundary-value problem with additional unknown functions in boundary conditions. These formula and problem were considered first by Lawruk \cite{Lawruk63a}. Note that the mentioned monographs \cite{Roitberg96, KozlovMazyaRossmann97} and survey \cite{Agranovich97} also examine elliptic problems for systems of differential equations and that the most general isomorphism theorem of Roitberg's type is proved by Kozhevnikov \cite{Kozhevnikov01} for pseudodifferential elliptic problems which form the Boutet de Monvel algebra. Roitberg's isomorphism theorem deals with normed function spaces which are certain modifications of Sobolev spaces. Such Sobolev--Roitberg spaces form two-sided scales of spaces and coincide with Sobolev spaces for the sufficiently large regularity index. The concept of Sobolev--Roitberg spaces proved to be fruitful not only for elliptic problems but also for parabolic and hyperbolic problems \cite{EidelmanZhitarashu98, Roitberg99}.

Although Sobolev spaces play a fundamental role in the modern theory of partial differential equations, the Sobolev scale is too course for various problems (see monographs \cite{Hermander63, Hermander83, MikhailetsMurach14, NicolaRodino10, Paneah00}). There is the necessity to use the classes of function spaces calibrated with the help of a function parameter, which characterizes the regularity of functions or distributions more finely than the number parameter used for the classical Sobolev or H\"older spaces. In 1963 H\"ormander \cite{Hermander63} introduced a broad and fruitful generalization of Sobolev spaces in this sense and gave applications of the spaces introduced to investigation of solvability of partial differential equations (see also his monograph \cite{Hermander83}). In the most interest case of Hilbert spaces, the H\"ormander space $\mathcal{B}_{2,\mu}$ consists of all tempered distributions $w$ in $\mathbb{R}^{n}$ such that $\mu\widehat{w}\in L_{2}(\mathbb{R}^{n})$
and is endowed with the norm $\|w\|_{\mathcal{B}_{2,\mu}}:=\|\mu\widehat{w}\|_{L_{2}(\mathbb{R}^{n})}$. Here, $\mu:\mathbb{R}^{n}\to(0,\infty)$ is a sufficiently general weight function, and $\widehat{w}$ is the Fourier transform of $w$. Of late decades H\"ormander spaces and their various generalizations---called the spaces of generalized smoothness---are actively investigated and applied to problems of mathematical analysis, to differential equations and stochastic processes (see monographs \cite{Jacob010205, MikhailetsMurach14, NicolaRodino10, Paneah00, Stepanets05, Triebel01} and references therein).

Recently Mikhailets and Murach \cite{MikhailetsMurach05UMJ5, MikhailetsMurach06UMJ2, MikhailetsMurach06UMJ3, MikhailetsMurach06UMJ11, MikhailetsMurach06UMB4, MikhailetsMurach07UMJ5, MikhailetsMurach08UMJ4} built a theory of solvability of regular elliptic boundary-value problems for H\"ormander spaces $H^{s,\varphi}:=\mathcal{B}_{2,\mu}$ and their modifications by Roitberg. In this theory, $\mu(\xi)=\langle\xi\rangle^{s}\varphi(\langle\xi\rangle)$ for arbitrary $\xi\in\mathbb{R}^{n}$, whereas $s$ is a real number and $\varphi:[1,\infty)\to(0,\infty)$ is a Borel measurable function  varying slowly in the sense of Karamata at infinity (as usual, $\langle\xi\rangle:=(1+|\xi|^{2})^{1/2}$). Every space $H^{s,\varphi}$ is attached to the Sobolev scale $\{H^{s}=H^{s,1}:s\in\mathbb{R}\}$ with the help of the number parameter $s$ and is obtained by the interpolation with a function parameter between the Sobolev spaces $H^{s-\varepsilon}$ and $H^{s+\delta}$ with $\varepsilon,\delta>0$. Using this interpolation method, Mikhailets and Murach extended all basic theorems on properties of regular elliptic problems from the Sobolev spaces over the indicated H\"ormander spaces. This theory also contains theorems on solvability of elliptic systems on manifolds in H\"ormander spaces \cite{Murach08MFAT2, Zinchenko17OpenMath}. It is set forth in monograph \cite{MikhailetsMurach14} and surveys \cite{09OperatorTheory191, MikhailetsMurach12BJMA2}. Nowadays this theory is extended \cite{AnopKasirenko16MFAT, AnopMurach14UMJ, AnopMurach14MFAT, MurachZinchenko13MFAT1, ZinchenkoMurach12UMJ11} over the class of all Hilbert spaces that are interpolation spaces between inner product Sobolev spaces; these interpolation spaces form a subclass of H\"ormander spaces $\mathcal{B}_{2,\mu}$  \cite{MikhailetsMurach13UMJ3, MikhailetsMurach15ResMath1}. Note that H\"ormander spaces and interpolation with a function parameter also find applications to parabolic initial-boundary value problems \cite{Los17UMJ9, LosMikhailetsMurach17CPAA, LosMurach13MFAT2, LosMurach17OpenMath}.

The goal of the present paper is to develop a version of this theory for nonregular elliptic problems with boundary conditions of higher orders with respect to the order of the corresponding elliptic differential equation. We build this version in the framework of the two-sided scale of H\"ormander spaces $H^{s,\varphi}$, with $-\infty<s<\infty$, modified in the sense of Roitberg. This modification was investigated in \cite{MikhailetsMurach08UMJ4}.

This paper consists of six sections. Section~\ref{sec1} is Introduction. In Section~\ref{sec2}, we formulate the elliptic problem under consideration and discuss the formally adjoint problem with respect to the special Green formula. In section~\ref{sec3}, we consider H\"ormander spaces and their modifications in the sense of Roitberg. Section~\ref{sec4} contains the main results of the paper. They are  theorems on the character of solvability of the elliptic problem on the two-sided scale of H\"ormander--Roitberg spaces and on local properties (up to the boundary) of its generalized solutions in these spaces. Among them is Theorem~\ref{3th2} on a complete collection of isomorphisms in H\"ormander--Roitberg spaces. As an application of these spaces, we give new sufficient conditions under which the generalized solutions have continuous partial derivatives of a prescribed order. Specifically, we obtain conditions for the generalized solutions to be classical. Section~\ref{sec5} is devoted to the method of interpolation with a function parameter between Hilbert spaces, which play a main role in the proof of the key Theorem~\ref{3th1}. The main results are proved in Section~\ref{sec6}.

\section{Statement of the problem}\label{sec2}

Let $\Omega$ be an open bounded domain in $\mathbb{R}^{n}$ with $n\geq2$. We suppose that its boundary $\Gamma:=\partial\Omega$ is an infinitely smooth closed manifold of dimension $n-1$, the $C^\infty$-structure on $\Gamma$ being induced by $\mathbb{R}^{n}$. Let $\nu(x)$ denote the unit vector of the inward normal to the boundary $\Gamma$ at a point $x\in\Gamma$.

We consider the boundary-value problem
\begin{gather}\label{3f1}
Au=f\quad\mbox{in}\quad\Omega,\\
B_{j}u=g_{j}\quad\mbox{on}\quad\Gamma,
\quad j=1,...,q.\label{3f2}
\end{gather}
Here,
$$
A:=A(x,D):=\sum_{|\mu|\leq 2q}a_{\mu}(x)D^{\mu}\
$$
is a linear differential operator on $\overline{\Omega}=\Omega\cup\Gamma$ of an arbitrary even order $2q\geq\nobreak2$. Besides, each
$$
B_{j}:=B_{j}(x,D):=\sum_{|\mu|\leq m_{j}}b_{j,\mu}(x)D^{\mu}\
$$
is a linear boundary differential operator on $\Gamma$ of an arbitrary order $m_{j}\geq0$. All the coefficients of these operators are complex-valued infinitely smooth functions given on $\overline{\Omega}$ and $\Gamma$ respectively. In the paper, all functions or distributions are supposed to be complex-valued, and hence all function spaces considered are complex.

Here and below, we use the standard designations:
$\mu:=(\mu_{1},\ldots,\mu_{n})$ is a multi-index, $|\mu|:=\mu_{1}+\cdots+\mu_{n}$,
$D^{\mu}:=D_{1}^{\mu_{1}}\cdots D_{n}^{\mu_{n}}$, $D_{k}:=i\partial/\partial x_{k}$ for each $k\in\{1,...,n\}$, where $i$ is imaginary unit and $x=(x_1,\ldots,x_n)$ is an arbitrary point in $\mathbb{R}^{n}$. We also put $D_{\nu}:=i\partial/\partial\nu(x)$.

Throughout the paper, we suppose that the boundary-value problem \eqref{3f1}, \eqref{3f2} is elliptic in the domain $\Omega$. This means that the differential operator $A$ is properly elliptic on $\overline{\Omega}$ and that the system of boundary operators $B:=(B_{1},\ldots,B_{q})$ satisfies Lopatinskii condition with respect to $A$ on $\Gamma$ (see, e.g., survey \cite[Section~1.2]{Agranovich97}).

\begin{example}\label{3ex1} \rm Consider a boundary-value problem that consists of the partial differential equation \eqref{3f1} (where the operator $A$ is properly elliptic on $\overline{\Omega}$) and the boundary conditions
\begin{equation*}
\frac{\partial^{k+j-1}u}{\partial\zeta^{k+j-1}}+\sum_{|\mu|< k+j-1}b_{j,\mu}(x)D^{\mu}=g_j\quad\text{on}\quad\Gamma,
\quad j=1,...,q.
\end{equation*}
Here, $0\leq k\in\mathbb{Z}$, whereas $\zeta:\Gamma\to\mathbb{R}^{n}$ is an infinitely smooth field of vectors $\zeta(x)$ which are nontangential to $\Gamma$ at every point $x\in\Gamma$. It is easy to verify that this boundary-value problem is elliptic in $\Omega$. If $k\leq q$, then it is regular elliptic (see, e.g., \cite[Subsection~5.2.1, Remark~4]{Triebel95}). Specifically, we may put $A:=\Delta^{q}$ and $\zeta(x):=\nu(x)$ for every $x\in\Gamma$. Here, as usual, $\Delta$ is the Laplace operator.
\end{example}

Henceforth we suppose that
\begin{equation*}
m:=\max\{m_{1},\ldots,m_{q}\}\geq2q.
\end{equation*}
Put $r:=m+1$ for the sake of convenience.

With the problem \eqref{3f1}, \eqref{3f2}, we associate the linear mapping
\begin{equation}\label{3f3}
\begin{gathered}
u\mapsto(Au,B_{1}u,...,B_{q}u)=(Au,Bu),\quad
\mbox{where}\quad u\in C^{\infty}(\overline{\Omega}).
\end{gathered}
\end{equation}
We will investigate properties of an extension by continuity of this mapping on appropriate pairs of Hilbert spaces introduced in the next section. To describe the range of this extension, we need the following special Green formula \cite[formula~(4.1.10)]{KozlovMazyaRossmann97}:
\begin{align*}
&(Au,v)_{\Omega}+\sum_{j=1}^{r-2q}(D_{\nu}^{j-1}Au,w_{j})_{\Gamma}+
\sum_{j=1}^{q}(B_{j}u,h_{j})_{\Gamma}\\
&=(u,A^{+}v)_{\Omega}+\sum_{k=1}^{r}\biggl(D_{\nu}^{k-1}u,K_{k}v+
\sum_{j=1}^{r-2q}R_{j,k}^{+}w_{j}+
\sum_{j=1}^{q}Q_{j,k}^{+}h_{j}\biggr)_{\Gamma}
\end{align*}
for arbitrary $u,v\in C^{\infty}(\overline{\Omega})$ and $w_{1},\ldots,w_{r-2q},h_{1},\ldots,h_{q}\in C^{\infty}(\Gamma)$. Here, $(\cdot,\cdot)_{\Omega}$ and $(\cdot,\cdot)_{\Gamma}$ denote the inner products in the Hilbert spaces $L_{2}(\Omega)$ and $L_{2}(\Gamma)$ of square integrable functions over $\Omega$ and $\Gamma$ respectively and later on denote the extension of these inner products by continuity. As usual, the differential operator $A^{+}$ is formally adjoint to $A$; namely,
$$
(A^{+}v)(x):=
\sum_{|\mu|\leq2q}D^{\mu}\bigl(\overline{a_{\mu}(x)}v(x)\bigr).
$$
Besides, all $R_{j,k}^{+}$ and $Q_{j,k}^{+}$ are the tangent differential operators which are adjoint respectively to $R_{j,k}$ and $Q_{j,k}$ with respect to $(\cdot,\cdot)_{\Gamma}$, with the  linear tangent differential operators $R_{j,k}:=R_{j,k}(x,D_{\tau})$ and $Q_{j,k}:=Q_{j,k}(x,D_{\tau})$ being taken from the representation of the boundary differential operators $D_{\nu}^{j-1}A$ and $B_{j}$ in the form
\begin{gather*}
D_{\nu}^{j-1}A(x,D)=\sum_{k=1}^{r}R_{j,k}(x,D_{\tau})D_{\nu}^{k-1},\quad
j=1,\ldots,r-2q,\\
B_{j}(x,D)=\sum_{k=1}^{r}Q_{j,k}(x,D_{\tau})D_{\nu}^{k-1},\quad
j=1,\ldots,q.
\end{gather*}
Note that $\mathrm{ord}\,R_{j,k}\leq 2q+j-k$ and $\mathrm{ord}\,Q_{j,k}\leq m_{j}-k+1$, with $R_{j,k}=0$ if $k\geq2q+j+1$ and with $Q_{j,k}=0$ if $k\geq m_{j}+2$. Finally, each $K_{k}:=K_{k}(x,D)$ is a certain linear boundary differential operator on $\Gamma$ of the order $\mathrm{ord}\,K_{k}\leq2q-k$ with coefficients from $C^{\infty}(\Gamma)$. Of course, if $k\geq2q+1$, then $K_{k}=0$.

Being based on the special Green formula, we consider the following boundary-value problem in $\Omega$ with $r-q$ additional unknown functions on $\Gamma$:
\begin{gather}\label{3f4}
A^{+}v=\omega\quad\mbox{in}\quad\Omega,\\
K_{k}v+\sum_{j=1}^{r-2q}R_{j,k}^{+}w_{j}+
\sum_{j=1}^{q}Q_{j,k}^{+}h_{j}=\theta_{k}\quad
\mbox{on}\quad\Gamma,\quad k=1,...,r.\label{3f5}
\end{gather}
Here, the function $v$ on $\overline{\Omega}$ and $r-q$ functions $w_{1},\ldots,w_{r-2q},h_{1},\ldots,h_{q}$ on $\Gamma$ are unknowns.
This problem is called formally adjoint to the problem \eqref{3f1}, \eqref{3f2} with respect to the special Green formula. The problem \eqref{3f1}, \eqref{3f2} is elliptic in $\Omega$ if and only if the formally adjoint problem \eqref{3f4}, \eqref{3f5} is elliptic in a relevant sense \cite[Theorem~4.1.1]{KozlovMazyaRossmann97}.

\section{H\"ormander spaces and their modifications\\in the sense of Roitberg}\label{sec3}

Following monograph \cite[Sections 1.3, 3.2, and 4.2]{MikhailetsMurach14}, we will consider H\"ormander spaces $H^{s,\varphi}$ and their modifications in the sense of Roitberg and discuss some of their properties. This spaces are parametrized with an arbitrary real number $s$ and function parameter $\varphi$ from the class $\mathcal{M}$.

By definition, the class $\mathcal{M}$ consists of all Borel measurable functions $\varphi:\nobreak[1,\infty)\rightarrow(0,\infty)$ that satisfy the following two conditions:
\begin{itemize}
\item[(i)] both the functions $\varphi$ and $1/\varphi$ are bounded on each compact interval $[1,b]$, with $1<b<\infty$;
\item[(ii)] $\varphi(\lambda t)/\varphi(t)\rightarrow 1$ as $t\rightarrow\infty$ for every  $\lambda>0$.
\end{itemize}
Property (ii) means that $\varphi$ is a slowly varying function at infinity in the sense of Karamata~\cite{Karamata30a}. Slowly varying functions are well investigated and have various applications  \cite{Seneta76, BinghamGoldieTeugels89}.

A standard example of a function $\varphi\in\mathcal{M}$ is a continuous function $\varphi:[1,\infty)\to(0,\infty)$ such that
\begin{equation*}
\varphi(t):=(\log t)^{r_{1}}(\log\log
t)^{r_{2}}\ldots(\underbrace{\log\ldots\log}_{k\;\mathrm{times}} t)^{r_{k}}\quad\mbox{for}\quad t\gg1.
\end{equation*}
Here, the integer $k\geq1$ and real numbers $r_{1},\ldots,r_{k}$ are  arbitrarily chosen.

This class admits the following description
$$
\varphi\in\mathcal{M}\;\;\Leftrightarrow\;\;
\varphi(t)=\exp\Biggl(\beta(t)+
\int\limits_{1}^{\:t}\frac{\gamma(\tau)}{\tau}\;d\tau\Biggr)\;\,
\mbox{for}\;\,t\geq1,
$$
where a bounded Borel measurable function $\beta:[1,\infty)\to\mathbb{R}$ has a finite limit at infinity, and a continuous function $\gamma:[1,\infty)\to\mathbb{R}$ converges to zero  at infinity. This description follows directly from Karamata's representation theorem (see, e.g., \cite[Section~1.2]{Seneta76}).

Let $s\in\mathbb{R}$ and $\varphi\in\mathcal{M}$. We first consider the H\"ormander space $H^{s,\varphi}$ over $\mathbb{R}^{n}$ with $n\geq1$ and then discuss its versions for $\Omega$ and $\Gamma$.

By definition, the linear space $H^{s,\varphi}(\mathbb{R}^{n})$ consists of all distributions $w\in\mathcal{S}'(\mathbb{R}^{n})$ that their Fourier transform $\widehat{w}$
is locally Lebesgue integrable over $\mathbb{R}^{n}$ and satisfies the condition
$$
\int\limits_{\mathbb{R}^{n}}
\langle\xi\rangle^{2s}\varphi^{2}(\langle\xi\rangle)\,
|\widehat{w}(\xi)|^{2}\,d\xi<\infty.
$$
Here, as usual, $\mathcal{S}'(\mathbb{R}^{n})$ denotes the linear topological space of all tempered distributions on $\mathbb{R}^{n}$, and
$\langle\xi\rangle:=(1+|\xi|^{2})^{1/2}$ is the smoothed modulus of $\xi\in\mathbb{R}^{n}$. The space $H^{s,\varphi}(\mathbb{R}^{n})$ is
endowed with the inner product
$$
(w_{1},w_{2})_{H^{s,\varphi}(\mathbb{R}^{n})}:=
\int\limits_{\mathbb{R}^{n}}
\langle\xi\rangle^{2s}\varphi^{2}(\langle\xi\rangle)\,
\widehat{w_{1}}(\xi)\,\overline{\widehat{w_{2}}(\xi)}\,d\xi
$$
and the corresponding norm
$$
\|w\|_{H^{s,\varphi}(\mathbb{R}^{n})}:=
(w,w)_{H^{s,\varphi}(\mathbb{R}^{n})}^{1/2}.
$$
This space is complete and separable with respect to this norm and is embedded continuously in $\mathcal{S}'(\mathbb{R}^{n})$. The set $C^{\infty}_{0}(\mathbb{R}^{n})$ of test functions is dense in $H^{s,\varphi}(\mathbb{R}^{n})$.

The space $H^{s,\varphi}(\mathbb{R}^{n})$ is an isotropic Hilbert case of the spaces $\mathcal{B}_{p,\mu}$ introduced and investigated by H\"ormander \cite[Section 2.2]{Hermander63} (see also his monograph \cite[Section 10.1]{Hermander83}). Namely, $H^{s,\varphi}(\mathbb{R}^{n})=\mathcal{B}_{p,\mu}$ if $p=2$ and $\mu(\xi)\equiv\langle\xi\rangle^{s}\varphi(\langle\xi\rangle)$. Note that the inner product H\"ormander spaces $\mathcal{B}_{2,\mu}$  coincide with the spaces introduced and investigated by Volevich and Paneah in \cite[Section~2]{VolevichPaneah65}.

In the case of $\varphi(t)\equiv1$, the space $H^{s,\varphi}(\mathbb{R}^{n})$ becomes the inner product Sobolev space $H^{s}(\mathbb{R}^{n})$ of order $s$. Generally,
\begin{equation}\label{3f7}
H^{s+\varepsilon}(\mathbb{R}^{n})\hookrightarrow H^{s,\varphi}(\mathbb{R}^{n})\hookrightarrow H^{s-\varepsilon}(\mathbb{R}^{n})
\quad\mbox{for every}\quad\varepsilon>0,
\end{equation}
with embeddings being continuous \cite[Lemma~1.5]{MikhailetsMurach14}.
They show that, the numeric parameter $s$ characterizes the main regularity of distributions $w\in H^{s,\varphi}(\mathbb{R}^{n})$, whereas the function parameter $\varphi$ sets certain supplementary regularity. Specifically, if $\varphi(t)\rightarrow\infty$ [or $\varphi(t)\rightarrow 0$] as $t\rightarrow\infty$, the parameter $\varphi$ will define supplementary positive [or negative] regularity. Thus, we can say that $\varphi$ refines the main regularity~$s$ in the class
$$
\{H^{s,\varphi}(\mathbb{R}^{n}):
s\in\mathbb{R},\varphi\in\mathcal{M}\}
$$
of function Hilbert spaces. This class is selected in \cite{MikhailetsMurach05UMJ5} and is called the refined Sobolev scale over $\mathbb{R}^{n}$ \cite[Section~1.3.3]{MikhailetsMurach14}.

Note \cite[Theorem~2.2.9]{Hermander63} that the spaces $H^{s,\varphi}(\mathbb{R}^{n})$ and $H^{-s,1/\varphi}(\mathbb{R}^{n})$ are mutually dual with respect to the extension by continuity of the inner product in $L_{2}(\mathbb{R}^{n})$. Here, the second space is well defined because $\varphi\in\mathcal{M}\Leftrightarrow1/\varphi\in\mathcal{M}$.

The Hilbert spaces $H^{s,\varphi}(\Omega)$ and $H^{s,\varphi}(\Gamma)$ are introduced in a standard way with the help of  $H^{s,\varphi}(\mathbb{R}^{n})$. Let us give the corresponding definitions.

By definition, the linear space $H^{s,\varphi}(\Omega)$ consists of the restrictions of all distributions $w\in H^{s,\varphi}(\mathbb{R}^{n})$ to the domain $\Omega$. The norm in $H^{s,\varphi}(\Omega)$ is defined by the formula
$$
\|u\|_{H^{s,\varphi}(\Omega)}:=
\inf\bigl\{\,\|w\|_{H^{s,\varphi}(\mathbb{R}^{n})}:\,
w\in H^{s,\varphi}(\mathbb{R}^{n}),\, w=u\;\,\mbox{in}\;\,\Omega\,\bigr\},
$$
where $u\in H^{s,\varphi}(\Omega)$. The space $H^{s,\varphi}(\Omega)$ is Hilbert and separable with respect to this norm and is embedded continuously in the linear topological space $\mathcal{D}'(\Omega)$ of all distributions on $\Omega$; besides, the set $C^{\infty}(\overline{\Omega})$ is dense in $H^{s,\varphi}(\Omega)$ \cite[Theorems 3.1 and 3.3(i)]{MikhailetsMurach14}.

Briefly saying, the linear space $H^{s,\varphi}(\Gamma)$ consists of all distributions on $\Gamma$ that yield elements of $H^{s,\varphi}(\mathbb{R}^{n-1})$ in local coordinates on $\Gamma$. Let us give a detailed definition. From the $C^{\infty}$-structure on $\Gamma$, we arbitrarily choose a finite atlas formed by certain local charts $\alpha_j: \mathbb{R}^{n-1}\leftrightarrow \Gamma_{j}$, where
$j=1,\ldots,\varkappa$. Here, the open sets $\Gamma_{1},\ldots,\Gamma_{\varkappa}$ form a covering of $\Gamma$. We also choose functions $\chi_j\in C^{\infty}(\Gamma)$, where
$j=1,\ldots,\varkappa$, that form a partition of unity on $\Gamma$ subject to the condition $\mathrm{supp}\,\chi_j\subset \Gamma_j$.

By definition, the linear space $H^{s,\varphi}(\Gamma)$ consists of all distributions $h\in\mathcal{D}'(\Gamma)$ such that $(\chi_{j}h)\circ\alpha_{j}\in H^{s,\varphi}(\mathbb{R}^{n-1})$ for each $j\in\{1,\ldots,\varkappa\}$. Here, of course, $\mathcal{D}'(\Gamma)$ is the linear topological space of all distributions on $\Gamma$, and $(\chi_{j}h)\circ\alpha_{j}$ denotes the representation of the distribution $\chi_{j}h$ in the local chart $\alpha_{j}$. The space $H^{s,\varphi}(\Gamma)$ is endowed with the norm
$$
\|h\|_{H^{s,\varphi}(\Gamma)}:=\biggl(\,\sum_{j=1}^{\varkappa}\,
\|(\chi_{j}h)\circ\alpha_{j}\|_{H^{s,\varphi}(\mathbb{R}^{n-1})}^{2}
\biggr)^{1/2}.
$$
This space is Hilbert and separable with respect to this norm and does not depend (up to equivalence of norms) on our choice of the atlas and the partition of unity on $\Gamma$ \cite[Theorem~2.21]{MikhailetsMurach14}. The space $H^{s,\varphi}(\Gamma)$ is embedded continuously in $\mathcal{D}'(\Gamma)$, and the set $C^{\infty}(\Gamma)$ is dense in $H^{s,\varphi}(\Gamma)$.

Note \cite[Theorem~2.3(v)]{MikhailetsMurach14} that the spaces $H^{s,\varphi}(\Gamma)$ and $H^{-s,1/\varphi}(\Gamma)$ are mutually dual (up to equivalence of norms) with respect to the form $(\cdot,\cdot)_{\Gamma}$, which is an extension by continuity of the inner product in $L_{2}(\Gamma)$. Specifically, the form $(h,v)_{\Gamma}$ is well defined for arbitrary $h\in H^{s,\varphi}(\Gamma)$ and $v\in C^{\infty}(\Gamma)$ and is equal to the value of the distribution $h\in\mathcal{D}'(\Gamma)$ on the test function $v$.

We have the classes of Hilbert spaces
\begin{equation}\label{3f8}
\{H^{s,\varphi}(\Omega):
s\in\mathbb{R},\varphi\in\mathcal{M}\}\quad\mbox{and}\quad
\{H^{s,\varphi}(\Gamma):
s\in\mathbb{R},\varphi\in\mathcal{M}\}.
\end{equation}
These classes respectively contain the inner product Sobolev spaces $H^{s}(\Omega):=H^{s,1}(\Omega)$ and $H^{s}(\Gamma):=H^{s,1}(\Gamma)$  of any order $s\in\mathbb{R}$. We have the dense compact  embeddings
\begin{equation}\label{3f10}
H^{s+\varepsilon,\varphi_{1}}(\Omega)\hookrightarrow H^{s,\varphi}(\Omega)\;\;\mbox{and}\;\;
H^{s+\varepsilon,\varphi_{1}}(\Gamma)\hookrightarrow H^{s,\varphi}(\Gamma)\;\;\mbox{whenever}\;\;\varepsilon>0,
\end{equation}
with $s\in\mathbb{R}$ and $\varphi,\varphi_{1}\in\mathcal{M}$ being arbitrary; see \cite[Theorems 2.3(iii) and 3.3(iii)]{MikhailetsMurach14}.

Discuss a connection between the scales \eqref{3f8}. Let $s>1/2$ and $\varphi\in\mathcal{M}$; then the trace mapping $u\mapsto u\!\upharpoonright\!\Gamma$, where $u\in C^{\infty}(\Gamma)$, extends uniquely (by continuity) to a bounded operator $R_{\Gamma}:H^{s,\varphi}(\Omega)\rightarrow H^{s-1/2,\varphi}(\Gamma)$, and this operator is surjective \cite[Theorem~3.5]{MikhailetsMurach14}. Thus, for every distribution $u\in H^{s,\varphi}(\Omega)$, its trace $R_{\Gamma}u$ on $\Gamma$ is well defined. But it is impossible to define this trace reasonably in the case where $s<1/2$. Namely, the above trace mapping cannot be extended to a continuous linear operator from the whole Sobolev space $H^{s}(\Omega)$ of order $s<\nobreak1/2$ to the linear topological space $\mathcal{D}'(\Gamma)$ (see, e.g., \cite[Remark~3.5]{MikhailetsMurach14}).

Hence, we cannot investigate the boundary-value problem \eqref{3f1}, \eqref{3f2} in the case where $u$ ranges over the whole space $H^{s,\varphi}(\Omega)$ with $s<m+1/2$. To study this problem for arbitrary real $s$, we have to use certain modifications of the H\"ormander space $H^{s,\varphi}(\Omega)$. We denote these modifications by $H^{s,\varphi,(k)}(\Omega)$, with $1\leq k\in\mathbb{Z}$, and introduce them by analogy with Roitberg's \cite{Roitberg64, Roitberg65} construction applied to Sobolev spaces (see also monographs \cite[Chapter~III, Section~6]{Berezansky68},  \cite[Section~2.1]{Roitberg96}, and survey \cite[Section~7.9]{Agranovich97}). Roitberg's approach was extended to the refined Sobolev scale by Mikhailets and Murach in \cite{MikhailetsMurach08UMJ4} and \cite[Section~4.2]{MikhailetsMurach14}.

Let, as above, $s\in\mathbb{R}$ and $\varphi\in\mathcal{M}$. We previously need to introduce a certain Hilbert space denoted by $H^{s,\varphi,(0)}(\Omega)$. Put  $H^{s,\varphi,(0)}(\Omega):=H^{s,\varphi}(\Omega)$ if $s\geq0$. But, if $s<0$, we let $H^{s,\varphi,(0)}(\Omega)$ denote the completion of $C^{\infty}(\overline{\Omega})$ with respect to the Hilbert norm
$$
\|u\|_{H^{s,\varphi,(0)}(\Omega)}:=
\sup\biggl\{\,\frac{|(u,v)_{\Omega}|}
{\;\quad\quad\quad\|v\|_{H^{-s,1/\varphi}(\Omega)}}\,:
\,v\in H^{-s,1/\varphi}(\Omega),\,v\neq0\biggr\}.
$$
Thus, given $s<0$, we have the Hilbert rigging of the space $L_{2}(\Omega)$ with the positive space $H^{-s,1/\varphi,(0)}(\Omega)=H^{-s,1/\varphi}(\Omega)$ and the negative space $H^{s,\varphi,(0)}(\Omega)$ (as to the notion and general properties of Hilbert riggings, see, e.g., \cite[Chapter~I, Section~1]{Berezansky68}). For every $s\in\mathbb{R}$, the spaces $H^{s,\varphi,(0)}(\Omega)$ and $H^{-s,1/\varphi,(0)}(\Omega)$ are mutually dual with respect to the form $(\cdot,\cdot)_{\Omega}$, which is an extension by continuity of the inner product in $L_{2}(\Omega)$ (if $s=0$, the duality is fulfilled up to equivalence of norms); see \cite[Theorem~3.9(iii)]{MikhailetsMurach14}. Specifically, the form $(u,v)_{\Omega}$ is well defined for arbitrary $u\in\nobreak H^{s,\varphi,(0)}(\Omega)$ and $v\in C^{\infty}(\overline{\Omega})$.

The negative space $H^{s,\varphi,(0)}(\Omega)$, where $s<0$, admits the following description. Consider the mapping $u\mapsto\mathcal{O}u$ where $u\in C^{\infty}(\overline{\Omega})$ and $(\mathcal{O}u)(x):=u(x)$ for $x\in\overline{\Omega}$ and $(\mathcal{O}u)(x):=0$ for $x\in\mathbb{R}^{n}\setminus\overline{\Omega}$. If $s<0$, this mapping extends uniquely (by continuity) to an isomorphism $\mathcal{O}$ between the Hilbert space $H^{s,\varphi,(0)}(\Omega)$ and the subspace
\begin{equation*}
\bigl\{v\in H^{s,\varphi}(\mathbb{R}^{n}):
\mathrm{supp}\,v\subseteq\overline{\Omega}\bigr\}
\end{equation*}
of $H^{s,\varphi}(\mathbb{R}^{n})$. This fact is a special case of the result proved in \cite[Section~3.4]{VolevichPaneah65} (see also \cite[Subsection~3.2.3]{MikhailetsMurach14}).

Now, let $1\leq k\in\mathbb{Z}$, and give a definition of the Hilbert space $H^{s,\varphi,(k)}(\Omega)$. Put
$$
E_{k}:=\{j-1/2: j\in\mathbb{Z},\;1\leq j\leq k\}.
$$
If $s\in\mathbb{R}\setminus E_{k}$, we let $H^{s,\varphi,(k)}(\Omega)$ denote the completion of $C^{\infty}(\overline{\Omega})$ with respect to the Hilbert norm
$$
\|u\|_{H^{s,\varphi,(k)}(\Omega)}:=
\biggl(\|u\|_{H^{s,\varphi,(0)}(\Omega)}^{2}+
\sum_{j=1}^{k}\;\|(D_{\nu}^{j-1}u)\!\upharpoonright\!\Gamma\|
_{H^{s-j+1/2,\varphi}(\Gamma)}^{2}\biggr)^{1/2}.
$$
If $s\in E_{k}$, we define the space $H^{s,\varphi,(k)}(\Omega)$ by means of the interpolation between Hilbert spaces. Namely, put
\begin{equation}\label{3fdef-interp}
H^{s,\varphi,(k)}(\Omega):=
\bigl[H^{s-1/2,\varphi,(k)}(\Omega),
H^{s+1/2,\varphi,(k)}(\Omega)\bigr]_{t^{1/2}}.
\end{equation}
A definition of this interpolation is given in Section~\ref{sec5}.

In the Sobolev case of $\varphi(t)\equiv1$, the space $H^{s,\varphi,(k)}(\Omega)$ was introduced by Roitberg
\cite{Roitberg64, Roitberg65}. Therefore we say that $H^{s,\varphi,(k)}(\Omega)$ is a modification of the H\"ormander space $H^{s,\varphi}(\Omega)$ in the sense of Roitberg or, briefly, is a H\"ormander--Roitberg space. The number $k$ is called the index of this modification. If $\varphi(t)\equiv1$, we will omit the index $\varphi$ in the designations of the spaces introduced in this section and below. Specifically, $H^{s,(k)}(\Omega):=H^{s,1,(k)}(\Omega)$ is a Sobolev--Roitberg space.

The space $H^{s,\varphi,(k)}(\Omega)$ with $s\in\mathbb{R}\setminus E_{k}$ admits the following description \cite[Theorem~4.11(i)]{MikhailetsMurach14}: the linear mapping
$$
T_{k}:u\mapsto\bigl(u,u\!\upharpoonright\!\Gamma,\ldots,
(D_{\nu}^{k-1}u)\!\upharpoonright\!\Gamma\bigr),\quad\mbox{where}\;\; u\in C^{\infty}(\overline{\Omega}),
$$
extends uniquely (by continuity) to an isometric linear operator
$$
T_{k}:H^{s,\varphi,(k)}(\Omega)\rightarrow
H^{s,\varphi,(0)}(\Omega)\oplus
\bigoplus_{j=1}^{k}\,H^{s-j+1/2,\varphi}(\Gamma)=:
\Pi_{s,\varphi,(k)}(\Omega,\Gamma),
$$
whose range consists of all vectors
$$
(u_{0},u_{1},\ldots,u_{k})\in\Pi_{s,\varphi,(k)}(\Omega,\Gamma)
$$
such that $u_{j}=R_{\Gamma}D_{\nu}^{j-1}u_{0}$ for each $j\in\{1,\ldots,k\}$ subject to $s>j-1/2$. If $s\in E_{k}$, this mapping extends uniquely to a bounded linear operator $T_{k}$ from $H^{s,\varphi,(k)}(\Omega)$ to $\Pi_{s,\varphi,(k)}(\Omega,\Gamma)$, but we cannot assert that this operator is isometric and that its range consists of all the vectors indicated above \cite[Remark~4.6]{MikhailetsMurach14}.

We have the class
\begin{equation*}
\bigl\{H^{s,\varphi,(k)}(\Omega):
s\in\mathbb{R},\varphi\in\mathcal{M}\bigr\}
\end{equation*}
of Hilbert spaces. They are separable
\cite[Theorem 4.12(i)]{MikhailetsMurach14}.
We may say that this class is a two-sided scale of H\"ormander--Roitberg spaces with respect to the number parameter $s$. Note that
\begin{equation}\label{3f11b}
H^{s,\varphi,(k)}(\Omega)=H^{s,\varphi}(\Omega)
\quad\mbox{for every}\;\;s>k-1/2
\end{equation}
up to equivalence of norms \cite[Theorem~4.12(iii)]{MikhailetsMurach14}.

We have the dense compact embedding
\begin{equation}\label{3f13}
H^{s+\varepsilon,\varphi_{1},(k)}(\Omega)\hookrightarrow H^{s,\varphi,(k)}(\Omega)\quad\mbox{whenever}\;\;\varepsilon>0,
\end{equation}
with $s\in\mathbb{R}$ and $\varphi,\varphi_{1}\in\mathcal{M}$ being arbitrary; see \cite[Theorem 4.12(iv)]{MikhailetsMurach14}. Specifically,
\begin{equation*}
H^{s+\varepsilon,(k)}(\Omega)\hookrightarrow H^{s,\varphi,(k)}(\Omega)\hookrightarrow H^{s-\varepsilon,(k)}(\Omega)\quad\mbox{whenever}\;\;\varepsilon>0.
\end{equation*}
Therefore we can put
\begin{equation*}
H^{-\infty,(k)}(\Omega):=
\bigcup_{s\in\mathbb{R},\,\varphi\in\mathcal{M}}H^{s,\varphi,(k)}(\Omega)
=\bigcup_{s\in\mathbb{R}}H^{s,(k)}(\Omega).
\end{equation*}
The linear space $H^{-\infty,(k)}(\Omega)$ is endowed with the topology of inductive limit. Note that the space $H^{-\infty,(k)}(\Omega)$ does not lie in  $\mathcal{D}'(\Omega)$. However, its part
\begin{equation*}
H^{k-1/2+}(\Omega):=\bigcup_{\substack{s>k-1/2,\\\varphi\in\mathcal{M}}}
H^{s,\varphi,(k)}(\Omega)=
\bigcup_{\substack{s>k-1/2,\\\varphi\in\mathcal{M}}}H^{s,\varphi}(\Omega)
\end{equation*}
lies in $\mathcal{D}'(\Omega)$. Here, the second equality is due to \eqref{3f11b}. Note that properties \eqref{3f11b} and \eqref{3f13} and, hence, these designations are also valid in $r=0$ case \cite[Theorem 3.9]{MikhailetsMurach14}.

The H\"ormander--Roitberg space $H^{s,\varphi,(k)}(\Omega)$ is suitable in the theory of boundary-value problems for arbitrary $s\in\mathbb{R}$ due to the following fact.

\begin{proposition}\label{3prop5}
Let an integer $k\geq1$. Suppose that $L$ is a linear differential operator on $\overline{\Omega}$ of order $\ell\leq k$ with coefficients from $C^{\infty}(\overline{\Omega})$. Then the mapping $u\mapsto Lu$, where $u\in C^{\infty}(\overline{\Omega})$, extends uniquely (by continuity) to a bounded linear operator
\begin{equation*}
L:\,H^{s,\varphi,(k)}(\Omega)\to H^{s-\ell,\varphi,(k-\ell)}(\Omega)
\quad\mbox{for all}\;\;s\in\mathbb{R}\;\;\mbox{and}\;\; \varphi\in\mathcal{M}.
\end{equation*}
Besides, suppose that $K$ is a boundary linear differential operator on $\Gamma$ of order $\varrho \leq k-1$ with coefficients from $C^{\infty}(\Gamma)$. Then the mapping $u\mapsto Ku$, where $u\in C^{\infty}(\overline{\Omega})$, extends uniquely (by continuity) to a bounded linear operator
\begin{equation*}
K:\,H^{s,\varphi,(k)}(\Omega)\to H^{s-\varrho-1/2,\varphi}(\Gamma)
\quad\mbox{for all}\;\;s\in\mathbb{R}\;\;\mbox{and}\;\; \varphi\in\mathcal{M}.
\end{equation*}
\end{proposition}

In the case of $\varphi(t)\equiv1$, this proposition is proved by Roitberg \cite[Lemma 2.3.1 and Corollary 2.3.1]{Roitberg96}. The general situation is derived from this case by interpolation with the help of Proposition~\ref{3prop1} given in Section~\ref{sec5} (cf. \cite[Proof of Theorem~4.13]{MikhailetsMurach14}). Note that the cases $\ell=0$ and $\varrho=0$ are admissible in Proposition~\ref{3prop5}. Specifically, the operator of multiplication by an arbitrary function from $C^{\infty}(\overline{\Omega})$ is bounded on every space $H^{s,\varphi,(k)}(\Omega)$.

\begin{remark}\label{3rem0}\rm
We also need the following version of the assertion of Proposition~\ref{3prop5} concerning the operator $L$: let the assumption of this proposition about $L$ be fulfilled; then the mapping $u\mapsto Lu$, where $u\in C^{\infty}(\overline{\Omega})$, extends uniquely (by continuity) to a bounded linear operator
\begin{equation*}
L:\,H^{s,\varphi,(k)}(\Omega)\to H^{s-\ell,\varphi,(q)}(\Omega)
\;\;\mbox{whenever}\;\;0\leq q\leq k-\ell\;\;\mbox{and}\;\; q\in\mathbb{Z}.
\end{equation*}
Here, $s\in\mathbb{R}$ and $\varphi\in\mathcal{M}$ are arbitrary. This version follows from Proposition~\ref{3prop5} and the inequality
\begin{equation}\label{rem3.2-estimate}
\|v\|_{H^{s-\ell,\varphi,(q)}(\Omega)}\leq c\,
\|v\|_{H^{s-\ell,\varphi,(k-\ell)}(\Omega)}\quad\mbox{for all}\quad
v\in C^{\infty}(\overline{\Omega}),
\end{equation}
which is used for $v:=Lu$. Here, $c$ is a certain positive number not depending on~$v$. If $s-\ell\notin E_{k-\ell}$, this inequality is evident, with $c=1$. For the rest~$s$, it is obtained by interpolation with the help of Proposition~\ref{3prop1} stated in Section~\ref{sec5}. Namely, let $\nobreak{s-\ell\in E_{k-\ell}}$. Since the identity mapping on $C^{\infty}(\overline{\Omega})$ extends uniquely to bounded operators
from $H^{s-\ell\mp1/4,(k-\ell)}(\Omega)$ to $H^{s-\ell\mp1/4,(q)}(\Omega)$, this mapping extends uniquely to a bounded operator from the space
\begin{equation*}
\bigl[H^{s-\ell-1/4,(k-\ell)}(\Omega),
H^{s-\ell+1/4,(k-\ell)}(\Omega)\bigr]_{\psi}=
H^{s-\ell,\varphi,(k-\ell)}(\Omega)
\end{equation*}
to the space
\begin{equation*}
\bigl[H^{s-\ell-1/4,(q)}(\Omega),H^{s-\ell+1/4,(q)}(\Omega)\bigr]_{\psi}=
H^{s-\ell,\varphi,(q)}(\Omega).
\end{equation*}
Here, $\psi$ is the interpolation parameter from Proposition~\ref{3prop1} in which $\varepsilon:=\delta:=1/4$. These equalities of spaces are valid up to equivalence of norms due to Proposition~\ref{3prop1}. This gives inequality \eqref{rem3.2-estimate} in the case of $s-\ell\in\nobreak E_{k-\ell}$.
\end{remark}

\section{Main results}\label{sec4}

Given $s\in\mathbb{R}$ and $\varphi\in\mathcal{M}$, we define the Hilbert space
\begin{equation*}
\mathcal{H}^{s-2q,\varphi,(r-2q)}(\Omega,\Gamma):=
H^{s-2q,\varphi,(r-2q)}(\Omega)\oplus\bigoplus_{j=1}^{q}
H^{s-m_j-1/2,\varphi}(\Gamma).
\end{equation*}
According to Proposition~\ref{3prop5}, mapping \eqref{3f3} extends uniquely (by continuity) to a bounded linear operator
\begin{equation}\label{3f14}
(A,B):H^{s,\varphi,(r)}(\Omega)\rightarrow
\mathcal{H}^{s-2q,\varphi,(r-2q)}(\Omega,\Gamma).
\end{equation}
The main results of the paper concern the properties of this operator, which corresponds to the elliptic boundary-value problem \eqref{3f1}, \eqref{3f2}. We will formulate them in this section.

Let $N$ denote the set of all solutions $u\in C^{\infty}(\overline{\Omega})$ to the boundary-value problem \eqref{3f1}, \eqref{3f2} in the homogeneous case where $f=0$ in $\Omega$ and each $g_{j}=0$ on $\Gamma$. Similarly, let $N_{\star}$ denote the set of all solutions
\begin{equation*}
(v,w_{1},\ldots,w_{r-2q},h_{1},\ldots,h_{q})\in C^{\infty}(\overline{\Omega})\times(C^{\infty}(\Gamma))^{r-q}
\end{equation*}
to the formally adjoint boundary-value problem \eqref{3f4}, \eqref{3f5}
in the homogeneous case where $\omega=0$ in $\Omega$ and each $\theta_{k}=0$ on $\Gamma$. Since both problems \eqref{3f1}, \eqref{3f2} and \eqref{3f4}, \eqref{3f5} are elliptic, both spaces $N$ and $N_{\star}$ are finite-dimensional \cite[Corollary~4.1.1]{KozlovMazyaRossmann97}.

\begin{theorem}\label{3th1}
The bounded operator \eqref{3f14} is Fredholm for arbitrary $s\in\mathbb{R}$ and $\varphi\in\mathcal{M}$. Its kernel is equal to $N$, and its range consists of all vectors
$$
(f,g_1,\ldots,g_q)\in\mathcal{H}^{s-2q,\varphi,(r-2q)}(\Omega,\Gamma)
$$
such that
\begin{equation}\label{3f15}
\begin{gathered}
(f_{0},v)_\Omega+\sum_{j=1}^{r-2q}(f_{j},w_{j})_{\Gamma}+
\sum_{j=1}^{q}(g_j,h_{j})_{\Gamma}=0 \\
\mbox{for every} \quad (v,w_{1},\ldots,w_{r-2q},h_{1},\ldots,h_{q})\in N_{\star},
\end{gathered}
\end{equation}
where $(f_{0},f_{1},\ldots,f_{r-2q}):=T_{r-2q}f$. The index of operator \eqref{3f14} equals $\dim N-\dim N_{\star}$ and therefore does not depend on $s$ and~$\varphi$.
\end{theorem}

In view of this theorem, we recall that a linear bounded operator $T:E_{1}\rightarrow E_{2}$, where $E_{1}$ and $E_{2}$ are Banach spaces, is called Fredholm if its kernel $\ker T$ and co-kernel $E_{2}/T(E_{1})$ are both finite-dimensional. If this operator is Fredholm, its range is closed in $E_{2}$ (see, e.g., \cite[Lemma~19.1.1]{Hermander85}), and its index
$$
\mathrm{ind}\,T:=\dim\ker T-\dim(E_{2}/T(E_{1}))
$$
is well defined and finite.

Note that if $s>m+1/2$, the Fredholm bounded operator \eqref{3f14} acts between the (non-modified) H\"ormander spaces
$$
(A,B):\,H^{s,\varphi}(\Omega)\rightarrow
H^{s-2q,\varphi}(\Omega)\oplus
\bigoplus_{j=1}^{q}H^{s-m_{j}-1/2,\varphi}(\Gamma).
$$
This is due to \eqref{3f11b} in view of $r=m+1$.

If $N=\{0\}$ and $N_{\star}=\{0\}$, operator \eqref{3f14} is an isomorphism between the spaces $H^{s,\varphi,(r)}(\Omega)$ and $\mathcal{H}^{s-2q,\varphi,(r-2q)}(\Omega,\Gamma)$. In the general situation, this operator induces an isomorphism between their certain subspaces of finite co-dimension. We build the isomorphism with the help of the next fact.

\begin{lemma}\label{3lem1}
For arbitrary $s\in\mathbb{R}$ and $\varphi\in\mathcal{M}$, we have the following decompositions of the spaces $H^{s,\varphi,(r)}(\Omega)$ and
$\mathcal{H}^{s-2q,\varphi,(r-2q)}(\Omega,\Gamma)$ in the direct sum of their subspaces:
\begin{equation}\label{3f16}
H^{s,\varphi,(r)}(\Omega)=
N\dotplus\bigl\{u\in H^{s,\varphi,(r)}(\Omega):
(u_{0},\omega)_\Omega=0\;\,\mbox{for all}\;\,\omega\in N\bigr\}
\end{equation}
and
\begin{equation}\label{3f17}
\begin{aligned}
&\mathcal{H}^{s-2q,\varphi,(r-2q)}(\Omega,\Gamma)\\
&=G\dotplus\bigl\{ (f,g_1,\ldots,g_q)\in\mathcal{H}^{s-2q,\varphi,(r-2q)}(\Omega,\Gamma):
\mbox{\eqref{3f15} is true}\bigr\}.
\end{aligned}
\end{equation}
Here, $u_{0}$ is the initial component of the vector $T_{r}u=(u_{0},u_{1},\ldots,u_{r})$. Besides, $G$ is a certain finite-dimensional space such that $\dim G=\dim N_{\star}$ and $G\subset C^{\infty}(\overline{\Omega})\times(C^{\infty}(\Gamma))^{q}$ and that $G$ does not depend on $s$ and~$\varphi$.
\end{lemma}

\begin{remark}\label{3rem1}\rm
Let $\varphi\in\mathcal{M}$ and assume that $s<1/2+2q$. Then
$T_{r-2q}$ is an isometric isomorphism between the spaces $H^{s-2q,\varphi,(r-2q)}(\Omega)$ and
$\Pi_{s-2q,\varphi,(r-2q)}(\Omega,\Gamma)$. Therefore we can choose
\begin{align*}
G:=\bigl\{&(T_{r-2q}^{-1}(v,w_{1},\ldots,w_{r-2q}),h_{1},\ldots,h_{q})\\
&:(v,w_{1},\ldots,w_{r-2q},h_{1},\ldots,h_{q})\in N_{\star}\bigr\}
\end{align*}
in decomposition~\eqref{3f17}. Indeed, the chosen space $G$ and the second summand in~\eqref{3f17} have the trivial intersection, and $\dim G=\dim N_{\star}$
is equal to the codimension of the second summand in view of Theorem~\ref{3th1}. However, this choice is impossible for $s\geq1/2+2q$ because the space
\begin{equation*}
\{(v,w_{1},\ldots,w_{r-2q}):
(v,w_{1},\ldots,w_{r-2q},h_{1},\ldots,h_{q})\in N_{\star}\}
\end{equation*}
does not lie generally in $T_{r-2q}(H^{s-2q,\varphi,(r-2q)}(\Omega))$ for these~$s$.
\end{remark}

Let $P$ and $Q$ respectively denote the projectors of the spaces
\begin{equation*}
H^{s,\varphi,(r)}(\Omega)\quad\mbox{and}\quad
\mathcal{H}^{s-2q,\varphi,(r-2q)}(\Omega,\Gamma)
\end{equation*}
onto the second summands in \eqref{3f16} and \eqref{3f17} parallel to the first summands. Evidently, these projectors are independent of $s$ and $\varphi$.

\begin{theorem}\label{3th2}
For arbitrary $s\in\mathbb{R}$ and $\varphi\in\mathcal{M}$, the restriction of mapping~\eqref{3f14} to the subspace
$P(H^{s,\varphi,(r)}(\Omega))$ is an isomorphism
\begin{equation}\label{3f18}
(A,B):\,P(H^{s,\varphi,(r)}(\Omega))\leftrightarrow
Q(\mathcal{H}^{s-2q,\varphi,(r-2q)}(\Omega,\Gamma)).
\end{equation}
\end{theorem}

This statement is a theorem on a complete collection of isomorphisms generated by the elliptic problem under consideration in H\"ormander--Roitberg spaces. This collection is complete in the sense that the number parameter $s$ ranges over the whole real axis.

Let us consider some important properties of the generalized solutions to the elliptic boundary-value problem \eqref{3f1}, \eqref{3f2}. Assume that
$$
(f,g):=(f,g_1,\ldots,g_q)\in H^{-\infty,(r-2q)}(\Omega)\times
(\mathcal{D}'(\Gamma))^{q}.
$$
A vector $u\in H^{-\infty,(r)}(\Omega)$ is called a (strong) generalized solution (in the sense of Roitberg) of this problem if $(A,B)u=(f,g)$, where $(A,B)$ is operator \eqref{3f14} for some parameters $s\in\mathbb{R}$ and $\varphi\in\mathcal{M}$. Of course, this definition is reasonable, i.e. it does not depend on $s$ and~$\varphi$. Besides, remark that
$$
H^{-\infty,(r-2q)}(\Omega)\times(\mathcal{D}'(\Gamma))^{q}=
\bigcup_{s\in\mathbb{R},\,\varphi\in\mathcal{M}}
\mathcal{H}^{s,\varphi,(r-2q)}(\Omega).
$$

\begin{theorem}\label{3th3}
Let $s\in\mathbb{R}$, $\varphi\in\mathcal{M}$, and a real number $\sigma>0$. Suppose that functions $\chi,\eta\in C^{\infty}(\overline{\Omega})$ satisfy the condition $\eta=1$ in a neighbourhood of \,$\mathrm{supp}\,\chi$. Then there exists a number $c=c(s,\varphi,\sigma,\chi,\eta)>0$ that an arbitrary vector
$u\in\nobreak H^{s,\varphi,(r)}(\Omega)$ satisfies the estimate
\begin{equation}\label{3f19}
\|\chi u\|_{H^{s,\varphi,(r)}(\Omega)}\leq c\,\bigl(\|\eta(A, B)u\|_{\mathcal{H}^{s-2q,\varphi,(r-2q)}(\Omega,
\Gamma)}+\|\eta u\|_{H^{s-\sigma,\varphi,(r)}(\Omega)}\bigr).
\end{equation}
Here, $c$ does not depend on $u$. If $\sigma\leq1$, we may replace $\eta$ with $\chi$ in the first summand on the right of \eqref{3f19}.
\end{theorem}

As to this theorem, we note first that the multiplication by $\eta\in C^{\infty}(\overline{\Omega})$ or $\eta\!\upharpoonright\!\Gamma$ is a bounded operator on every space $H^{\theta,\varphi,(k)}(\Omega)$ or $H^{\theta,\varphi}(\Gamma)$, where $\theta\in\mathbb{R}$, $\varphi\in\mathcal{M}$, and $1\leq k\in\mathbb{Z}$; see \cite[Theorem 4.13(i) and Lemma~2.5]{MikhailetsMurach14}. Hence, all terms of \eqref{3f19} are well defined and finite, we naturally interpreting $\eta(A, B)u$ as $(\eta Au,(\eta\!\upharpoonright\!\Gamma)B_{1}u,\ldots,
(\eta\!\upharpoonright\!\Gamma)B_{q}u)$.

In the case of $\chi(\cdot)\equiv\eta(\cdot)\equiv1$, inequality \eqref{3f19} is a global \textit{a priori} estimate of the generalized solution $u$ to the elliptic boundary-value problem \eqref{3f1}, \eqref{3f2}. In the general situation, this inequality is a local \textit{a priori} estimate (up to $\Gamma$) of~$u$. Indeed, for every nonempty open subset of $\overline{\Omega}$, we can choose functions $\chi$ and $\eta$ which satisfy the condition of Theorem~\ref{3th3} and whose supports lie in this subset.

Now we pay our attention to a local regularity up to $\Gamma$ of the generalized solutions. Let $V$ be an arbitrary open subset of $\mathbb{R}^{n}$ that has a nonempty intersection with the domain~$\Omega$. We set $\Omega_0:=\Omega\cap V$ and $\Gamma_{0}:=\Gamma\cap V$ (the case of $\Gamma_{0}=\emptyset$ is possible). Let us introduce local analogs of the spaces $H^{s,\varphi,(k)}(\Omega)$ and $H^{s,\varphi}(\Gamma)$, with $s\in\mathbb{R}$, $\varphi\in\mathcal{M}$, and $1\leq k\in\mathbb{Z}$. By definition, the space $H^{s,\varphi,(k)}_{\mathrm{loc}}(\Omega_{0},\Gamma_{0})$ consists of all vectors $u\in H^{-\infty,(k)}(\Omega)$ such that $\chi u\in H^{s,\varphi,(k)}(\Omega)$ for an arbitrary function $\chi\in C^{\infty}(\overline{\Omega})$ with $\mathrm{supp}\,\chi\subset\Omega_0\cup\Gamma_{0}$. Similarly, the space $H^{s,\varphi}_{\mathrm{loc}}(\Gamma_{0})$ consists of all distributions  $h\in\nobreak\mathcal{D}'(\Gamma)$ such that $\chi h\in H^{s,\varphi}(\Gamma)$ for every function $\chi\in C^{\infty}(\Gamma)$ with $\mathrm{supp}\,\chi\subset\Gamma_{0}$.
Put
\begin{equation*}
\mathcal{H}^{s-2q,\varphi,(r-2q)}_{\mathrm{loc}}(\Omega_{0},\Gamma_{0}):=
H^{s-2q,\varphi,(r-2q)}_{\mathrm{loc}}(\Omega_{0},\Gamma_{0})\times
\prod_{j=1}^{q}H^{s-m_{j}-1/2,\varphi}_{\mathrm{loc}}(\Gamma_{0}).
\end{equation*}

\begin{theorem}\label{3th4}
Suppose that $u\in H^{-\infty,(r)}(\Omega)$ is a generalized solution to the elliptic boundary-value problem \eqref{3f1}, \eqref{3f2} in which
\begin{equation*}
(f,g)\in
\mathcal{H}^{s-2q,\varphi,(r-2q)}_{\mathrm{loc}}(\Omega_{0},\Gamma_{0})
\end{equation*}
for some parameters $s\in\mathbb{R}$ and $\varphi\in\mathcal{M}$. Then $u\in H^{s,\varphi,(r)}_{\mathrm{loc}}(\Omega_{0},\Gamma_{0})$.
\end{theorem}

If $\Omega_{0}=\Omega$ and $\Gamma_{0}=\Gamma$, then the equalities  $H^{s,\varphi,(r)}_{\mathrm{loc}}(\Omega_{0},\Gamma_{0})=
H^{s,\varphi,(r)}(\Omega)$ and $\mathcal{H}^{s-2q,\varphi,(r-2q)}_{\mathrm{loc}}(\Omega_{0},\Gamma_{0})=
\mathcal{H}^{s-2q,\varphi,(r-2q)}(\Omega,\Gamma)$ hold true. Therefore, Theorem~\ref{3th4} states in this case that the regularity of $u$ increases globally, i.e. on the whole closed domain~$\overline{\Omega}$. If $\Gamma_{0}=\emptyset$, this theorem becomes an assertion about increase in local smoothness of $u$ in neighbourhoods of inner points of $\overline{\Omega}$.

In the Sobolev case of $\varphi(t)\equiv1$, Theorems \ref{3th1}, \ref{3th2}, \ref{3th3}, and \ref{3th4} are proved in Roitberg's monographs \cite[Chapter 4 and Section~7.2]{Roitberg96}. However, Roitberg does not use the formally adjoint problem \eqref{3f4}, \eqref{3f5} for the description of the range of the Fredholm operator~\eqref{3f14}. This is done in Kozlov, Maz'ya and Rossmann's monograph \cite[Section~4.1]{KozlovMazyaRossmann97} for elliptic problems with additional unknown functions in boundary conditions (see also \cite[Subsection~2.7.4]{Roitberg99}). Note that Isomorphism Theorem~\ref{3th2} in the Sobolev case was established by Kostarchuk and Roitberg \cite[Theorem~5]{KostarchukRoitberg73} (without involving the formally adjoint problem). Apparently, Vainberg and Grushin \cite[Section~4]{VainbergGrushin67b}  showed first that it is necessary to use the expression of the form
\begin{equation*}
\sum_{j=1}^{r-2q}(f_{j},w_{j})_{\Gamma}
\end{equation*}
in the description \eqref{3f15} of the domain of $(A,B)$ in the $m\geq2q$ case under consideration. Elliptic problems with additional unknown functions in boundary conditions, specifically the formally adjoint problem \eqref{3f4}, \eqref{3f5}, are investigated in H\"ormander spaces in \cite{ChepurukhinaMurach15MFAT, MurachChepurukhina15UMJ}, the orders of boundary conditions being less than $2q$.

As an application of Theorem~\ref{3th4}, we will obtain sufficient conditions under which the generalized solution $u\in H^{-\infty,(r)}(\Omega)$ yields a function from $C^p(\Omega_{0}\cup\Gamma_{0})$ for some integer $p\geq0$. It is naturally to admit the following agreement for $u\in H^{-\infty,(r)}(\Omega)$: in the case of $\Gamma_{0}\neq\emptyset$ we write $u\in C^{p}(\Omega_{0}\cup\Gamma_{0})$ if an only if $\chi u\in H^{r-1/2+}(\Omega)\cap C^p(\overline{\Omega})$ for every function $\chi\in C^{\infty}(\overline{\Omega})$ such that $\mathrm{supp}\,\chi\subset\Omega_{0}\cup\Gamma_{0}$. Note the inclusion $\chi u\in H^{r-1/2+}(\Omega)$ ensures that $\chi u$ is a distribution in $\Omega$, for which the condition $\chi u\in C^p(\overline{\Omega})$ makes sense. Besides, considering the case of $\Gamma_{0}=\emptyset$, we write $u\in C^{p}(\Omega_{0})$ if and only if $(\chi u)_{0}\in H^{-1/2+}(\Omega)\cap C^{p}(\Omega)$ for every function $\chi\in C^{\infty}(\overline{\Omega})$ such that $\mathrm{supp}\,\chi\subset\Omega_{0}$. Here, $(\chi u)_{0}$ denotes the initial component of the vector $T_{r}(\chi u)$, whereas $T_{r}$ is the operator introduced in Section~\ref{sec3}.

\begin{theorem}\label{3th5}
Let $0\leq p\in\mathbb{Z}$, and assume that $p>r-(n+1)/2$ in the case of $\Gamma_{0}\neq\emptyset$. Suppose that the condition of Theorem~$\ref{3th4}$ is fulfilled for $s:=p+n/2$ and a certain function $\varphi\in\mathcal{M}$ subject to
\begin{equation}\label{3f20a}
\int\limits_{1}^{\infty}\frac{dt}{t\,\varphi^{2}(t)}<\infty.
\end{equation}
Then $u\in C^{p}(\Omega_{0}\cup\Gamma_{0})$.
\end{theorem}

\begin{remark}\label{3rem3} \rm
Condition \eqref{3f20a} is sharp in Theorem~\ref{3th5}. This means the following: let $0\leq p\in\mathbb{Z}$ and $\varphi\in\mathcal{M}$; then it follows from the implication
\begin{equation}\label{3-implication}
\begin{aligned}
&\bigl(u\in H^{-\infty, (r)}(\Omega),\;(A,B)u\in
\mathcal{H}^{p+n/2-2q,\varphi,(r-2q)}_{\mathrm{loc}}(\Omega_{0},\Gamma_{0})
\bigr)\\
&\Longrightarrow\;u\in C^{p}(\Omega_{0}\cup\Gamma_{0})
\end{aligned}
\end{equation}
that $\varphi$ satisfies condition \eqref{3f20a}.
\end{remark}

If we stated a version of Theorem \ref{3th5} in the framework of Sobolev--Roitberg spaces, we would suppose that
\begin{equation*}
(f,g)\in\mathcal{H}^{s-2q,(r-2q)}_{\mathrm{loc}}(\Omega_{0},\Gamma_{0})
\quad\mbox{for some}\quad s>p+n/2.
\end{equation*}
This assumption is stronger than that made in Theorem~\ref{3th5} in terms of H\"ormander--Roitberg spaces.

We also give a sufficient condition under which the generalized solution $u$ to problem \eqref{3f1}, \eqref{3f2} is classical, i.e., $u\in C^{2q}(\Omega)\cap C^{m}(U_{\sigma}\cup \Gamma)$ for some $\sigma>0$, where $U_{\sigma}:=\bigl\{ x\in \Omega:\, \mathrm{dist}(x,\Gamma)<\sigma \bigr\}$. If the solution $u$ is classical, the left-hand sides of this problem are calculated with the help of classical partial derivatives and are continuous functions on $\Omega$ and $\Gamma$ respectively.

\begin{theorem}\label{3th6}
Suppose that a vector $u\in H^{-\infty,(r)}(\Omega)$ is a generalized solution to the elliptic boundary-value problem \eqref{3f1}, \eqref{3f2} in which
\begin{gather}\label{3f23}
f\in H^{n/2,\varphi_1,(r-2q)}_{\mathrm{loc}}(\Omega,\emptyset)\cap H^{m-2q+n/2,\varphi_2,(r-2q)}(U_{\sigma},\Gamma),\\
g_j\in H^{m-m_j+(n-1)/2,\varphi_2}(\Gamma),\quad\mbox{with}\quad
j=1,\ldots,q\label{3f24},
\end{gather}
for certain parameters $\varphi_{1},\varphi_{2}\in\mathcal{M}$ that satisfy condition \eqref{3f20a} with $\varphi:=\varphi_{1}$ and $\varphi:=\varphi_{2}$ respectively. Then $u$ is a classical solution.
\end{theorem}

We will prove Theorems \ref{3th1}, \ref{3th2}--\ref{3th5}, \ref{3th6} and Lemma~\ref{3lem1} and justify Remark~\ref{3rem3} in Section~\ref{sec6}.

\section{Interpolation between Hilbert spaces}\label{sec5}

The Hilbert spaces considered in Section~\ref{sec3} have an important interpolation property, which will play a key role in our proof of Theorem~1. Namely, every H\"ormander space $H^{s,\varphi}(G)$, where $s\in\mathbb{R}$, $\varphi\in\mathcal{M}$, and $G\in\{\mathbb{R}^{n},\Omega,\Gamma\}$, can be obtained by the interpolation with an appropriate function parameter between the Sobolev spaces $H^{s-\varepsilon}(G)$ and $H^{s+\delta}(G)$ with $\varepsilon,\delta>0$. A similar result holds true for the H\"ormander--Roitberg spaces $H^{s,\varphi,(k)}(\Omega)$.

Therefore we will recall the definition of the interpolation with a function parameter between Hilbert spaces. This interpolation method was introduced by Foia\c{s} and Lions \cite[p.~278]{FoiasLions61}. We restrict the presentation to the case of separable Hilbert spaces, which will be enough for our purposes. We follow monograph \cite[Section~1.1]{MikhailetsMurach14}.

Let $X:=[X_{0},X_{1}]$ be an ordered pair of separable complex Hilbert spaces that satisfy the following two conditions: $X_{1}$ is a dense manifold in $X_{0}$, and there is a number $c>0$ such that $\|w\|_{X_{0}}\leq c\,\|w\|_{X_{1}}$ for every $w\in X_{1}$ (briefly saying, the dense continuous embedding $X_{1}\hookrightarrow X_{0}$ holds). This pair is called admissible. As is known \cite[Chapter~1, Section~1]{LionsMagenes72}, for $X$ there exists a positive-definite self-adjoint operator $J$ on $X_{0}$ with the domain $X_{1}$ such that $\|Jw\|_{X_{0}}=\|w\|_{X_{1}}$ for every $w\in X_{1}$. This operator is uniquely determined by the pair $X$ and is called a generating operator for~$X$.

Let $\mathcal{B}$ denote the set of all Borel measurable functions
$\psi:(0,\infty)\rightarrow(0,\infty)$ such that $\psi$ is bounded on every compact interval $[a,b]$ with $0<a<b<\infty$ and that $1/\psi$ is bounded on every closed semiaxis $[r,\infty)$ with $r>0$.

Given $\psi\in\mathcal{B}$, we consider the operator $\psi(J)$, which is defined (and positive-definite) in $X_{0}$ as the Borel function $\psi$ of the positive-definite self-adjoint operator~$J$. Let $[X_{0},X_{1}]_{\psi}$ or, simply, $X_{\psi}$ denote the domain of the operator $\psi(J)$ endowed with the inner product
$(w_1, w_2)_{X_\psi}:=(\psi(J)w_1,\psi(J)w_2)_{X_0}$ and the
corresponding norm $\|w\|_{X_\psi}=(w,w)_{X_\psi}^{1/2}$. The space $X_{\psi}$ is Hilbert and separable. The continuous and dense embedding $X_\psi \hookrightarrow X_0$ holds true.

A function $\psi\in\mathcal{B}$ is called an interpolation parameter if and only if the following condition is fulfilled for all admissible pairs $X=[X_{0},X_{1}]$ and $Y=[Y_{0},Y_{1}]$ of Hilbert spaces and for an arbitrary linear mapping $T$ given on $X_{0}$: if the restriction of $T$ to $X_{j}$ is a bounded operator $T:X_{j}\rightarrow Y_{j}$ for each $j\in\{0,1\}$, the restriction of $T$ to $X_{\psi}$ is also a bounded operator $T:X_{\psi}\rightarrow Y_{\psi}$. If $\psi$ is an interpolation parameter, we say that the Hilbert space $X_{\psi}$ is obtained by the interpolation with the function parameter $\psi$ of the pair $X=[X_{0},X_{1}]$ (or between $X_{0}$ and $X_{1}$). We also say that the bounded operator $T:X_{\psi}\rightarrow Y_{\psi}$ is obtained by the interpolation with the parameter  $\psi$ of the bounded operators $T:X_{0}\rightarrow Y_{0}$ and $T:X_{1}\rightarrow Y_{1}$. In this case, the dense and continuous embeddings $X_{1}\hookrightarrow X_{\psi}\hookrightarrow X_{0}$ hold.

Note that a function $\psi\in\mathcal{B}$ is an interpolation parameter if and only if $\psi$ is pseudoconcave in a neighbourhood of $+\infty$ \cite[Section~1.1.9]{MikhailetsMurach14}. The latter condition means that there exists a concave function $\psi_{1}:(b,\infty)\rightarrow(0,\infty)$, where $b\gg1$, such that both functions $\psi/\psi_{1}$ and $\psi_{1}/\psi$ are bounded on $(b,\infty)$. This fundamental result follows from Peetre's theorem \cite{Peetre66, Peetre68} about  description of all interpolation functions of positive order (see also \cite[Section~5.4]{BerghLefstrem80}).

Specifically, the power function $\psi(t)\equiv t^{\theta}$ is an interpolation parameter if and only if $0\leq\theta\leq1$. In this case the exponent $\theta\in[0,1]$ is considered as a number parameter of the interpolation (see \cite[Chapter~IV, Section~1, Subsection~10]{KreinPetuninSemenov82} or \cite[Chapter~1, Sections 2 and 5]{LionsMagenes72}). This method of interpolation between spaces was historically first and belongs to Lions and S.~Krein. It is used in  definition \eqref{3fdef-interp} of certain H\"ormander--Roitberg spaces.

Let us formulate the above-mentioned interpolation property of the H\"ormander spaces and their modifications in the sense of Roitberg.

\begin{proposition}\label{3prop1}
Let a function $\varphi\in\mathcal{M}$ and positive real numbers $\varepsilon$ and $\delta$ be given. Define a function $\psi\in\mathcal{B}$ by the formula
\begin{gather*}
\psi(t)=\begin{cases}
t^{\varepsilon/(\varepsilon+\delta)}\varphi(t^{1/(\varepsilon+\delta)})&  \text{if}\;\;t\geq 1\\
\varphi(1)&\text{if}\;\;0<t<1.
\end{cases}
\end{gather*}
Then $\psi$ is an interpolation parameter, and
$$
\bigl[H^{s-\varepsilon}(G),H^{s+\delta}(G)\bigr]_{\psi}=H^{s,\varphi}(G)
\quad\mbox{for every}\quad s\in\mathbb{R}
$$
with equivalence of norms. Here, $G\in\{\mathbb{R}^{n},\Omega,\Gamma\}$; if $G=\mathbb{R}^{n}$, the equality of norms holds. Besides, let an integer $k\geq1$, and assume that at least one of the inequalities $s-\varepsilon>k-1/2$ and $s+\delta<k+1/2$ is satisfied in the case of odd $k$. Then
\begin{equation*}
\bigl[H^{s-\varepsilon,(k)}(\Omega),
H^{s+\delta,(k)}(\Omega)\bigr]_{\psi}
=H^{s,\varphi,(k)}(\Omega)\quad\mbox{for every}\quad s\in\mathbb{R}
\end{equation*}
with equivalence of norms.
\end{proposition}

This proposition is a collection of the results proved in \cite[Theorems 1.14, 2.2, 3.2, and 4.22]{MikhailetsMurach14}. We also need the following two general properties of the interpolation \cite[Theorems 1.7 and 1.5]{MikhailetsMurach14}.

\begin{proposition}\label{3prop3}
Let $X=[X_0,X_1]$ and  $Y=[Y_0,Y_1]$ be two admissible pairs of  Hilbert spaces. Suppose that a linear mapping $T$ is given on $X_0$ and satisfies the following condition: the restrictions of $T$ to the spaces $X_0$ and $X_1$ are Fredholm bounded operators $T:X_0\rightarrow Y_0$ and $T:X_1\rightarrow Y_1$ respectively, and these operators  have the same kernel and the same index. Then, for an arbitrary interpolation parameter $\psi\in\mathcal{B}$, the restriction of $T$ to $X_\psi$ is also a Fredholm bounded operator $T:X_\psi\rightarrow Y_\psi$ with the same kernel and the same index; the range of this operator is equal to $Y_\psi\cap T(X_0)$.
\end{proposition}

\begin{proposition}\label{3prop4}
Let $[X_{0}^{(k)},X_{1}^{(k)}]$, where $k=1,\ldots,p$, be a finite
collection of admissible pairs of Hilbert spaces. Then, for every function $\psi\in\mathcal{B}$, we have
\begin{equation*}
\biggl[\,\bigoplus_{k=1}^{p}X_{0}^{(k)},\,
\bigoplus_{k=1}^{p}X_{1}^{(k)}\biggr]_{\psi}=\,
\bigoplus_{k=1}^{p}\bigl[X_{0}^{(k)},\,X_{1}^{(k)}\bigr]_{\psi}
\end{equation*}
with equality of norms.
\end{proposition}

\section{Proofs of the main results}\label{sec6}

In this section we will prove the results formulated in Section~\ref{sec4}.

\begin{proof}[Proof of Theorem~$\ref{3th1}$.] In the Sobolev case of $\varphi(t)\equiv1$, this theorem is proved in Roitberg's monograph \cite[Theorem 4.1.3]{Roitberg96} excepting the indication of the relation of $N_{\star}$ to the formally adjoint problem \eqref{3f4}, \eqref{3f5}. Roitberg states only that $N_{\star}$ in \eqref{3f15} is a certain finite-dimensional space which lies in $C^{\infty}(\overline{\Omega})\times(C^{\infty}(\Gamma))^{r-q}$ and does not depend on~$s\in\mathbb{R}$. It follows from this that
\begin{equation}\label{range-Sobolev-case}
\begin{gathered}
(A,B)(H^{s,(r)}(\Omega))=
\mathcal{H}^{s-2q,(r-2q)}(\Omega,\Gamma)\cap
(A,B)(H^{s-\varepsilon,(r)}(\Omega))\\
\mbox{for arbitrary}\;\;s\in\mathbb{R}\;\;\mbox{and}\;\;\varepsilon>0.
\end{gathered}
\end{equation}
The indicated description of the range of the Fredholm operator \eqref{3f14} with the help of the formally adjoint problem \eqref{3f4}, \eqref{3f5} is given in Kozlov, Maz'ya and Rossmann's monograph \cite[Theorem~4.1.4]{KozlovMazyaRossmann97} on the additional assumption that $\nobreak{s\in\mathbb{Z}}$. In the case of fractional $s$, this description follows directly from property \eqref{range-Sobolev-case} considered for $s-\varepsilon\in\mathbb{Z}$. Thus, the conclusion of Theorem~\ref{3th1} is true in the Sobolev case for every real~$s$.

Assuming $\varphi\in\mathcal{M}$, we will deduce Theorem~\ref{3th1} from the Sobolev case with the help of the interpolation with a function parameter. Namely, let $s\in\mathbb{R}$ and $\varphi\in\mathcal{M}$. Choose a number $\varepsilon\in(0,1/2)$, and consider the bounded linear operators
\begin{equation}\label{3f25}
(A,B):H^{s\mp\varepsilon,(r)}(\Omega)\rightarrow
\mathcal{H}^{s\mp\varepsilon-2q,(r-2q)}(\Omega,\Gamma).
\end{equation}
Owing to Theorem~\ref{3th1} in the Sobolev case, they are Fredholm, have the common kernel $N$ and the same index $\dim N-\dim N_{\star}$. Besides,
\begin{equation}\label{3f26}
(A,B)(H^{s\mp\varepsilon,(r)}(\Omega))=
\bigl\{(f,g)\in\mathcal{H}^{s\mp\varepsilon-2q,(r-2q)}(\Omega,\Gamma):\,
\eqref{3f15}\;\mbox{is true}\bigr\}.
\end{equation}

Let $\psi$ be the interpolation parameter from Proposition~\ref{3prop1} where $\delta:=\varepsilon$. Applying the interpolation with the function parameter $\psi$ to operators \eqref{3f25} and using Proposition~\ref{3prop3}, we obtain the Fredholm bounded operator
\begin{equation}\label{3f27}
\begin{aligned}
(A,B)&:\bigl[H^{s-\varepsilon,(r)}(\Omega),
H^{s+\varepsilon,(r)}(\Omega)\bigr]_{\psi}\\
&\to\bigl[\mathcal{H}^{s-\varepsilon-2q,(r-2q)}(\Omega,\Gamma),
\mathcal{H}^{s+\varepsilon-2q,(r-2q)}(\Omega,\Gamma)\bigr]_{\psi}.
\end{aligned}
\end{equation}
This operator is a restriction of the mapping \eqref{3f25} defined on
$H^{s-\varepsilon,(r)}(\Omega)$.

Let us describe the interpolation spaces in \eqref{3f27} with the help of Proposition~\ref{3prop1}. Since $0<\varepsilon<1/2$, both numbers
$s\mp\varepsilon$ satisfy at least one of the inequalities $s\mp\varepsilon>r-1/2$ and $s\mp\varepsilon<r+1/2$. Hence, owing to
Proposition~\ref{3prop1}, we have the equality
\begin{equation*}
\bigl[H^{s-\varepsilon,(r)}(\Omega),
H^{s+\varepsilon,(r)}(\Omega)\bigr]_{\psi}=H^{s,\varphi,(r)}(\Omega).
\end{equation*}
Besides,
\begin{align*}
&\bigl[\mathcal{H}^{s-\varepsilon-2q,(r-2q)}(\Omega,\Gamma),
\mathcal{H}^{s+\varepsilon-2q,(r-2q)}(\Omega,\Gamma)\bigr]_{\psi}\\
&=\bigl[H^{s-2q-\varepsilon,(r-2q)}(\Omega),
H^{s-2q+\varepsilon,(r-2q)}(\Omega)\bigr]_{\psi}\\
&\quad\oplus\bigoplus_{j=1}^{q}\,
\bigl[H^{s-m_{j}-1/2-\varepsilon}(\Gamma),
H^{s-m_{j}-1/2+\varepsilon}(\Gamma)\bigr]_{\psi}\\
&=H^{s-2q,\varphi,(r-2q)}(\Omega)\oplus
\bigoplus_{j=1}^{q}H^{s-m_{j}-1/2,\varphi}(\Gamma)\\
&=\mathcal{H}^{s-2q,\varphi,(r-2q)}(\Omega,\Gamma)
\end{align*}
in view of Proposition~\ref{3prop4}. These equalities of spaces are fulfilled with equivalence of norms.

Hence, the Fredholm bounded operator \eqref{3f27} is operator \eqref{3f14}. According to Proposition~\ref{3prop3}, the kernel and index of \eqref{3f14} coincide respectively with the common kernel $N$ and index $\dim N-\dim N_{\star}$ of operators \eqref{3f25}. Moreover, owing to \eqref{3f26}, the range of \eqref{3f14} is equal to
\begin{align*}
&\mathcal{H}^{s-2q,\varphi,(r-2q)}(\Omega,\Gamma)\cap
(A,B)(H^{s-\varepsilon,(r)}(\Omega))\\
&=\bigl\{(f,g)\in\mathcal{H}^{s-2q,\varphi,(r-2q)}(\Omega,\Gamma):
\eqref{3f15}\;\mbox{is true}\bigr\}.
\end{align*}
Thus, the Fredholm operator \eqref{3f14} has all the properties stated in Theorem~\ref{3th1}.
\end{proof}

\begin{proof}[Proof of Lemma $\ref{3lem1}$.] Decomposition \eqref{3f16} is demonstrated in the same way as that used in the proof of \cite[Lemma~4.4, formula (4.90)]{MikhailetsMurach14}, where the case of even $r$ is considered (see also the proofs of \cite[Lemmas 4.1.1 and 4.1.2]{Roitberg96} in the Sobolev case). Let us establish decomposition \eqref{3f17}.

We first examine the case where $s=r$ and $\varphi(t)\equiv1$. Consider the orthogonal sum
\begin{equation}\label{3f-proof1}
\mathcal{H}^{r-2q,(r-2q)}(\Omega,\Gamma)=
G_{0}\oplus\bigl\{(f,g)\in\mathcal{H}^{r-2q,(r-2q)}(\Omega,\Gamma):
\mbox{\eqref{3f15} is true}\bigr\}.
\end{equation}
Here, $G_{0}$ is a finite-dimensional subspace of $\mathcal{H}^{r-2q,(r-2q)}(\Omega,\Gamma)$ due to Theorem~\ref{3th1}. Since the linear manifold  $C^{\infty}(\overline{\Omega})\times(C^{\infty}(\Gamma))^{q}$ is dense in $\mathcal{H}^{r-2q,(r-2q)}(\Omega,\Gamma)$, it follows from \eqref{3f-proof1} by \cite[Lemma~2.1]{GohbergKrein60} that there exists a finite-dimensional subspace $G$ of $C^{\infty}(\overline{\Omega})\times(C^{\infty}(\Gamma))^{q}$ that
formula \eqref{3f17} holds true in the case considered. Owing to Theorem~\ref{3th1}, we have the equality $\dim G=\dim N_{\star}$.

Decomposition \eqref{3f17} remains true for this subspace $G$ in the general situation of arbitrary $s\in\mathbb{R}$ and $\varphi\in\mathcal{M}$. Indeed, since $G\subset C^{\infty}(\overline{\Omega})\times(C^{\infty}(\Gamma))^{q}$, we have the equalities
\begin{align*}
&G\cap\bigl\{(f,g)\in\mathcal{H}^{s-2q,\varphi,(r-2q)}(\Omega,\Gamma):
\mbox{\eqref{3f15} is true}\bigr\}\\
&=G\cap\bigl\{(f,g)\in\mathcal{H}^{r-2q,(r-2q)}(\Omega,\Gamma):
\mbox{\eqref{3f15} is true}\bigr\}=\{0\}.
\end{align*}
Besides, owing to Theorem~\ref{3th1}, the dimension of $G$ equaled $\dim N_{\star}$ coincides with the codimension of the second summand in \eqref{3f17}.
\end{proof}

\begin{proof}[Proof of Theorem $\ref{3th2}$.]
By virtue of Theorem~\ref{3th1}, the bounded linear operator \eqref{3f18} is a bijection. Therefore, this operator is an isomorphism according to the Banach theorem on inverse operator.
\end{proof}

\begin{proof}[Proof of Theorem $\ref{3th3}$.]
In the case of $\chi(\cdot)\equiv\eta(\cdot)\equiv1$, this theorem follows from Theorem~\ref{3th1} and Petre's lemma \cite[Lemma~3]{Peetre61}. Namely, according to this lemma, estimate \eqref{3f19} in this case is a consequence of the facts that the bounded operator \eqref{3f14} has a finite-dimensional kernel and closed range and that the embedding $H^{s,\varphi,(r)}(\Omega)\hookrightarrow H^{s-\sigma,\varphi,(r)}(\Omega)$ is compact. Thus, there exists a number  $c=c(s,\varphi,\sigma)>0$ such that
\begin{equation}\label{3f30}
\|u\|_{H^{s,\varphi,(r)}(\Omega)}\leq c\,
\bigl(\|(A, B)u\|_{\mathcal{H}^{s-2q,\varphi,(r-2q)}(\Omega,\Gamma)}+
\|u\|_{H^{s-\sigma,\varphi,(r)}(\Omega)}\bigr)
\end{equation}
for arbitrary $u\in H^{s,\varphi,(r)}(\Omega)$.

Let us prove Theorem~$\ref{3th3}$ in the general situation. Owing to \eqref{3f13}, we may restrict ourselves to the case of $1\leq\sigma\in\mathbb{Z}$. In this case, we will prove the theorem by induction in~$\sigma$.

Let us deduce Theorem~$\ref{3th3}$ in the $\sigma=1$ case from the global estimate \eqref{3f30}. Choose a vector $u\in H^{s,\varphi,(r)}(\Omega)$ arbitrarily. Replacing $u$ with $\chi u$ in \eqref{3f30} and taking $\sigma:=1$, we get the estimate
\begin{equation}\label{3f31}
\|\chi u\|_{H^{s,\varphi,(r)}(\Omega)}\leq c_0\,\bigl(\|(A, B)(\chi u)\|_{\mathcal{H}^{s-2q,\varphi,(r-2q)}(\Omega,\Gamma)}+
\|\chi u\|_{H^{s-1,\varphi,(r)}(\Omega)}\bigr).
\end{equation}
Here, the positive number $c_0:=c(s,\varphi,1)$ does not depend on $u$ (and $\chi$).

Interchanging the operator of the multiplication by $\chi$ with the differential operators $A$ and $B_1,\ldots,B_q$, we write the equalities
\begin{equation}\label{3f32}
\begin{aligned}
(A,B)(\chi u)&=(A,B)(\chi\eta u)=\chi(A,B)(\eta u)+(A',B')(\eta u)\\
&=\chi(A,B)u+(A',B')(\eta u).
\end{aligned}
\end{equation}
Here, $A'$ is a linear differential operator on $\overline{\Omega}$ of order $\mathrm{ord}\,A'\leq 2q-1$, whereas $B':=(B_{1}',\ldots,B_{q}')$ is a collection of boundary linear differential operators on $\Gamma$ such that $\mathrm{ord}\,B_{j}'\leq m_{j}-1$ for each $j\in\{1,\ldots,q\}$. All the coefficients of $A'$ and each $B_{j}'$ are infinitely smooth on $\overline{\Omega}$ and $\Gamma$ respectively. Being evident for $u\in C^{\infty}(\overline{\Omega})$, equalities \eqref{3f32} extend over all $u\in H^{s,\varphi,(r)}(\Omega)$ by closure due to Proposition~\ref{3prop5} and Remark~\ref{3rem0}, with all terms of these equalities being considered as elements of $\mathcal{H}^{s-2q,\varphi,(r-2q)}(\Omega,\Gamma)$.

Specifically,
\begin{equation}\label{3f32a}
\|(A',B')(\eta u)\|_{\mathcal{H}^{s-2q,\varphi,(r-2q)}(\Omega,\Gamma)}
\leq c_{1}\|\eta u\|_{H^{s-1,\varphi,(r)}(\Omega)}.
\end{equation}
Besides,
\begin{equation}\label{3f32b}
\|\chi u\|_{H^{s-1,\varphi,(r)}(\Omega)}=
\|\chi\eta u\|_{H^{s-1,\varphi,(r)}(\Omega)}\leq
c_{2}\|\eta u\|_{H^{s-1,\varphi,(r)}(\Omega)}.
\end{equation}
Here and below in the proof, we let $c_{1},c_{2},\ldots,c_{7}$ denote some positive numbers that do not depend on~$u$. Applying formulas \eqref{3f32}--\eqref{3f32b} to \eqref{3f31}, we obtain the inequalities
\begin{equation}\label{3f32c}
\begin{aligned}
\|\chi u\|_{H^{s,\varphi,(r)}(\Omega)}&\leq c_0\,
\bigl(\|\chi(A,B)u\|_{\mathcal{H}^{s-2q,\varphi,(r-2q)}(\Omega,\Gamma)}\\
&\quad\;\;+\|(A',B')(\eta u)\|_{\mathcal{H}^{s-2q,\varphi,(r-2q)}(\Omega,\Gamma)}+
\|\chi u\|_{H^{s-1,\varphi,(r)}(\Omega)}\bigr)\\
&\leq c_3\,
\bigl(\|\chi(A,B)u\|_{\mathcal{H}^{s-2q,\varphi,(r-2q)}(\Omega,\Gamma)}+
\|\eta u\|_{H^{s-1,\varphi,(r)}(\Omega)}\bigr).
\end{aligned}
\end{equation}
Thus, we have justified the last sentence of Theorem~$\ref{3th3}$. Estimate \eqref{3f19} in the $\sigma=1$ case follows from \eqref{3f32c} due to the inequality
\begin{align*}
\|\chi(A, B)u\|_{\mathcal{H}^{s-2q,\varphi,(r-2q)}(\Omega,\Gamma)}&=
\|\chi \eta(A, B)u\|_{\mathcal{H}^{s-2q,\varphi,(r-2q)}(\Omega,\Gamma)}
\\&\leq c_4
\|\eta(A, B)u\|_{\mathcal{H}^{s-2q,\varphi,(r-2q)}(\Omega,
\Gamma)}.
\end{align*}

We now choose an integer $\lambda\geq1$ arbitrarily and assume that  Theorem~$\ref{3th3}$ holds in the case of $\sigma=\lambda$. Let us deduce this theorem in the case of $\sigma=\lambda+1$. Let, as above, the functions $\chi,\eta$ satisfy the assumption of this theorem. We can choose a function $\eta_{1}\in C^{\infty}(\overline{\Omega})$ such that $\eta_{1}=1$ in a neighbourhood of $\mathrm{supp}\,\chi$ and that $\eta=1$ in a neighbourhood of $\mathrm{supp}\,\eta_{1}$. According to our assumption, an arbitrary vector $u\in H^{s,\varphi,(r)}(\Omega)$ satisfies the estimate
\begin{equation*}
\|\chi u\|_{H^{s,\varphi,(r)}(\Omega)}\leq c_{5}\bigl(\|\eta_{1}(A, B)u\|_{\mathcal{H}^{s-2q,\varphi,(r-2q)}(\Omega,\Gamma)}
+\|\eta_{1}u\|_{H^{s-\lambda,\varphi,(r)}(\Omega)}\bigr).
\end{equation*}
Here,
\begin{align*}
\|\eta_{1}(A,B)u\|_{\mathcal{H}^{s-2q,\varphi,(r-2q)}(\Omega,\Gamma)}&=
\|\eta_{1}\eta(A,B)u\|_{\mathcal{H}^{s-2q,\varphi,(r-2q)}(\Omega,\Gamma)}
\\&\leq c_6\|\eta(A, B)u\|_{\mathcal{H}^{s-2q,\varphi,(r-2q)}(\Omega,
\Gamma)}.
\end{align*}
Besides,
\begin{equation*}
\|\eta_{1}u\|_{H^{s-\lambda,\varphi,(r)}(\Omega)}\leq c_{7}
\bigl(\|\eta(A,B)u\|_
{\mathcal{H}^{s-\lambda-2q,\varphi,(r-2q)}(\Omega,\Gamma)}+\|\eta u\|_{H^{s-\lambda-1,\varphi,(r)}(\Omega)}\bigr)
\end{equation*}
due to Theorem~$\ref{3th3}$ in the $\sigma=1$ case just proved. These three bounds immediately imply the required estimate \eqref{3f19} in the case of $\sigma=\lambda+1$. By induction, Theorem~\ref{3th3} is proved for every integer $\sigma\geq1$.
\end{proof}

\begin{proof}[Proof of Theorem $\ref{3th4}$.]
First we will prove Theorem~\ref{3th4} in the case where $\Omega_{0}=\Omega$ and $\Gamma_{0}=\Gamma$. By Theorem~\ref{3th1}, the vector $(f,g):=(A,B)u$ satisfies condition \eqref{3f15}. Besides, $(f,g)\in\mathcal{H}^{s-2q,\varphi,(r-2q)}(\Omega,\Gamma)$ by the hypothesis of Theorem $\ref{3th4}$ in the case considered. Therefore, $(f,g)\in (A,B)(H^{s,\varphi,(r)}(\Omega))$ due to Theorem~\ref{3th1}. Thus, there exists a vector $v\in H^{s,\varphi,(r)}(\Omega)$ such that $(A,B)v=(f,g)$. Then $(A,B)(u-v)=0$; hence, $w:=u-v\in N\subset C^{\infty}(\overline{\Omega})$ by Theorem~\ref{3th1}. Therefore, $u=v+w\in H^{s,\varphi,(r)}(\Omega)$. Theorem \ref{3th4} is proved in the case considered.

Let us now prove Theorem~\ref{3th4} in the general situation. We will previously show that, under the hypothesis of this theorem, the following implication holds for every integer $k\geq1$:
\begin{equation}\label{th4.6-proof-f1}
u\in H^{s-k,\varphi,(r)}_{\mathrm{loc}}(\Omega_{0},\Gamma_{0})
\;\Longrightarrow\;
u\in H^{s-k+1,\varphi,(r)}_{\mathrm{loc}}(\Omega_{0},\Gamma_{0}).
\end{equation}
We arbitrarily choose an integer $k\geq1$ and suppose that the premise of this implication is true. Then $\chi u\in H^{s-k,\varphi,(r)}(\Omega)$ for every function $\chi\in C^{\infty}(\overline{\Omega})$ such that $\mathrm{supp}\,\chi\subset\Omega_{0}\cup\Gamma_{0}$. Choose a function $\eta\in C^{\infty}(\overline{\Omega})$ such that $\eta=\nobreak1$ in a neighbourhood of $\mathrm{supp}\,\chi$. Owing to \eqref{3f32}, we have the equality
\begin{equation}\label{th4.6-proof-f2}
(A,B)(\chi u)=\chi(A,B)u+(A',B')(\eta u),
\end{equation}
in which the differential operators $A'$ and $B'$ are the same as that in the proof of Theorem~~\ref{3th3}. Here,
\begin{equation}\label{th4.6-proof-f3}
\chi(A,B)u=\chi(f,g)\in\mathcal{H}^{s-2q,\varphi,(r-2q)}(\Omega,\Gamma)
\end{equation}
due to the hypothesis of Theorem \ref{3th4}. Besides, $\eta u\in H^{s-k,\varphi,(r)}(\Omega)$ by the premise, which implies the inclusion
\begin{equation}\label{th4.6-proof-f4}
(A',B')(\eta u)\in\mathcal{H}^{s-k-2q+1,\varphi,(r-2q)}(\Omega,\Gamma)
\end{equation}
due to Proposition~\ref{3prop5} and Remark~\ref{3rem0}. Owing to formulas \eqref{th4.6-proof-f2}--\eqref{th4.6-proof-f4}, we have the inclusion
\begin{equation}
(A,B)(\chi u)\in\mathcal{H}^{s-k+1-2q,\varphi,(r-2q)}(\Omega,\Gamma).
\end{equation}
Hence, $\chi u\in H^{s-k+1,\varphi,(r)}(\Omega)$ according to Theorem~\ref{3th4} proved above in the global case of $\Omega_{0}=\Omega$ and $\Gamma_{0}=\Gamma$. Thus, $u\in H^{s-k+1,\varphi,(r)}_{\mathrm{loc}}(\Omega_{0},\Gamma_{0})$ because of the arbitrariness of the function $\chi$ used. Implication \eqref{th4.6-proof-f1} is proved.

Using this implication, we can prove Theorem~\ref{3th4} in the general case. According to the condition $u\in H^{-\infty,(r)}(\Omega)$ and embedding \eqref{3f13}, there exists an integer $\ell\geq1$ such that
\begin{equation*}
u\in H^{s-\ell,\varphi,(r)}(\Omega)\subset H^{s-\ell,\varphi,(r)}_{\mathrm{loc}}(\Omega_{0},\Gamma_{0}).
\end{equation*}
Therefore, using \eqref{th4.6-proof-f1} successively for $k=\ell$, $k=\ell-1$, ..., and $k=1$, we arrive at the required inclusion $u\in H^{s,\varphi,(r)}_{\mathrm{loc}}(\Omega_{0},\Gamma_{0})$.
\end{proof}

To prove Theorem $\ref{3th5}$, we make use of the following corollary from H\"ormander's embedding theorem \cite[Theorem~2.2.7]{Hermander63}: let $0\leq p\in\mathbb{Z}$ and $\varphi\in\mathcal{M}$; then
\begin{equation}\label{3f20}
\int\limits_{1}^{\infty}\frac{dt}{t\,\varphi^{2}(t)}<\infty\;\;
\Longleftrightarrow\;\;
H^{p+n/2,\varphi}(\Omega)\hookrightarrow C^p(\overline{\Omega}),
\end{equation}
the embedding being compact (see \cite[Theorem~3.4]{MikhailetsMurach14}). Property \eqref{3f20}, so to say, refines the well-known Sobolev's embedding theorem, which asserts that
\begin{equation*}
s>p+n/2\;\;\Longleftrightarrow\;\;
H^{s}(\Omega)\hookrightarrow C^p(\overline{\Omega}).
\end{equation*}

\begin{proof}[Proof of Theorem $\ref{3th5}$.]
By Theorem~\ref{3th4}, we have the inclusion
\begin{equation}\label{th4.7-proof-f1}
u\in H^{p+n/2,\varphi,(r)}_{\mathrm{loc}}(\Omega_{0},\Gamma_{0}).
\end{equation}
We first consider the case of $\Gamma_{0}\neq\emptyset$. Then $p>r-(n+1)/2$ by the hypothesis of Theorem~\ref{3th5}. Owing to property \eqref{th4.7-proof-f1}, equality \eqref{3f11b}, condition \eqref{3f20a}, and equivalence \eqref{3f20}, we obtain the inclusion
\begin{equation*}
\chi u\in H^{p+n/2,\varphi,(r)}(\Omega)=H^{p+n/2,\varphi}(\Omega)
\subset H^{r-1/2+}(\Omega)\cap C^{p}(\overline{\Omega})
\end{equation*}
for every function $\chi\in C^{\infty}(\overline{\Omega})$ such that $\mathrm{supp}\,\chi\subset\Omega_0\cup\Gamma_0$. Thus, $u\in C^{p}(\Omega_{0}\cup\Gamma_{0})$.

Turn to the case of $\Gamma_{0}=\emptyset$. According to  \eqref{th4.7-proof-f1}, \eqref{3f20a}, and \eqref{3f20}, we have the inclusion
\begin{equation*}
(\chi u)_{0}\in H^{p+n/2,\varphi,(0)}(\Omega)=H^{p+n/2,\varphi}(\Omega)
\subset H^{-1/2+}(\Omega)\cap C^{p}(\overline{\Omega})
\end{equation*}
for every function $\chi\in C^{\infty}(\overline{\Omega})$ such that $\mathrm{supp}\,\chi\subset\Omega_0$. Thus, $u\in C^{p}(\Omega_{0})$.
\end{proof}

\begin{proof}[Proof of Theorem $\ref{3th6}$.] Theorem~\ref{3th5} in the case of $p:=2q$, $\varphi:=\varphi_1$, $\Omega_{0}:=\Omega$, and $\Gamma_{0}:=\emptyset$ asserts that condition \eqref{3f23} implies the inclusion $u\in C^{2q}(\Omega)$. Besides, owing to the same theorem in the case of $p:=m$, $\varphi:=\varphi_2$, $\Omega_{0}:=U_{\sigma}$, and $\Gamma_{0}:=\Gamma$, we conclude that conditions \eqref{3f23} and \eqref{3f24} entail the inclusion $u\in C^{m}(U_{\sigma}\cup \Gamma)$.
Note that the assumption $p>r-(n+1)/2$ made in this theorem in the $\Gamma_{0}\neq\emptyset$ case is fulfilled for $p=m$. Thus, the solution $u$ is classical.
\end{proof}

Completing this section, we will justify Remark $\ref{3rem3}$. Let  $0\leq p\in\mathbb{Z}$ and $\varphi\in\mathcal{M}$, and assume that implication \eqref{3-implication} holds true. We must show that $\varphi$ satisfies \eqref{3f20a}. Choose an open ball $V$ subject to the condition $\overline{V}\subset\Omega_{0}$, and make use of a bounded linear operator
\begin{equation*}
\Upsilon:H^{p+n/2,\varphi}(V)\to H^{p+n/2,\varphi,(r)}(\Omega)
\end{equation*}
such that $(\Upsilon v)_{0}=v$ in $V$ for every $v\in H^{p+n/2,\varphi}(V)$. Here, as above, $(\Upsilon v)_{0}$ is the initial component of the vector $T_{r}(\Upsilon v)$, with $T_{r}$ being the operator introduced in Section~\ref{sec3}. We will build $\Upsilon$ in the next paragraph. Choosing $v\in H^{p+n/2,\varphi}(V)$ arbitrarily, we put $u:=\Upsilon v$ in implication \eqref{3-implication}. Since the vector $u\in H^{p+n/2,\varphi,(r)}(\Omega)$ satisfies the premise of this implication, the inclusion $u\in C^{p}(\Omega_{0}\cup\Gamma_{0})$ holds true. Specifically, $(\chi u)_{0}\in C^{p}(\Omega)$ for every function $\chi\in C^{\infty}(\overline{\Omega})$ such that $\mathrm{supp}\,\chi\subset\Omega_{0}$ and that $\chi=1$ on $\overline{V}$. Observing that
\begin{equation*}
v=(\Upsilon v)_{0}=u_{0}=\chi u_{0}=(\chi u)_{0}\quad
\mbox{in}\;\;V,
\end{equation*}
we arrive at the inclusion $v\in C^p(\overline{V})$. Thus, $H^{p+n/2,\varphi}(V)\hookrightarrow C^p(\overline{V})$, which implies the required condition \eqref{3f20a} due to equivalence \eqref{3f20} in which $V$ is taken instead of $\Omega$. We have justified Remark~\ref{3rem3}.

Let us build the operator $\Upsilon$. To this end we need a linear mapping $K:L_{2}(V)\to L_{2}(\mathbb{R}^{n})$ such that $Kv=v$ in $V$ for every $v\in L_{2}(V)$ and that the restriction of this mapping to every Sobolev space $H^{\sigma}(V)$ of order $\sigma\in[0,p+n]$ is a bounded operator from $H^{\sigma}(V)$ to $H^{\sigma}(\mathbb{R}^{n})$. This extension mapping is given, e.g., in \cite[Theorem~4.2.3]{Triebel95}. We also choose a function $\eta\in C^{\infty}_{0}(\mathbb{R}^{n})$ such that $\mathrm{supp}\,\eta\subset\Omega$ and $\eta=1$ on $\overline{V}$. Put
\begin{equation*}
\Upsilon v:=T_{r}^{-1}\bigl((\eta Kv)\!\!\upharpoonright\!\Omega,\underbrace{0,\ldots,0}_
{r\;\text{times}}\,)\bigr)\quad\mbox{for every}\quad v\in L_{2}(V).
\end{equation*}
The element $\Upsilon v\in H^{0,(r)}(\Omega)$ is well defined because the vector $((\eta Kv)\!\!\upharpoonright\!\Omega,0,\ldots,0)$ belongs to $T_{r}(H^{0,(r)}(\Omega))=\Pi_{0,(r)}(\Omega,\Gamma)$. Hence, we have the linear mapping $v\mapsto \Upsilon v$ from  $H^{0}(V)=L_{2}(V)$ to $H^{0,(r)}(\Omega)$. It follows from the definition of this mapping that $(\Upsilon v)_{0}=v$ in $V$ for every $v\in L_{2}(V)$. Besides, the restriction of the mapping to $H^{\sigma}(\Omega)$ is a bounded operator
\begin{equation}\label{rem-proof-f3}
\Upsilon:H^{\sigma}(V)\to H^{\sigma,(r)}(\Omega)
\quad\mbox{for every}\quad\sigma\in[0,p+n]\setminus E_{r}.
\end{equation}
This follows from the implication
\begin{equation*}
v\in H^{\sigma}(V)\;\Longrightarrow\;
((\eta Kv)\!\!\upharpoonright\!\Omega,0,\ldots,0)\in T_{r}^{-1}(H^{\sigma,(r)}(\Omega)),
\end{equation*}
which holds because the function $\eta Kv\in H^{\sigma}(\mathbb{R}^{n})$ vanishes near $\Gamma$. Applying the interpolation with a function parameter to operators \eqref{rem-proof-f3}, we can prove that the mapping  $\Upsilon$ acts continuously from $H^{s,\varphi}(V)$ to $H^{s,\varphi,(r)}(\Omega)$ for every $s\in(0,p+n)$ and $\varphi\in\mathcal{M}$. Namely, choose a number $\varepsilon\in(0,1/4)$ such that $s\mp\varepsilon\in[0,p+n]\setminus E_{r}$, and let $\psi$ be the interpolation parameter from Proposition~\ref{3prop1} in the case of $\delta=\varepsilon$. Applying the interpolation with the function parameter $\psi$ to the bounded linear operators \eqref{rem-proof-f3} considered for $\sigma=s\mp\varepsilon$, we conclude by Proposition~\ref{3prop1} that a restriction of the mapping $\Upsilon$ is a bounded operator
\begin{align*}
\Upsilon&:H^{s,\varphi}(V)=
\bigl[H^{s-\varepsilon}(V),H^{s+\varepsilon}(V)\bigr]_{\psi}\\
&\to\bigl[H^{s-\varepsilon,(r)}(\Omega),
H^{s+\varepsilon,(r)}(\Omega)\bigr]=H^{s,\varphi,(r)}(\Omega).
\end{align*}
Thus, the required operator $\Upsilon$ is built.


\begin{thebibliography}{99}

\bibitem{AgmonDouglisNirenberg59}
{S. Agmon, A. Douglis, and L. Nirenberg},
\textit{Estimates near the boundary for solutions of elliptic partial differential equations satisfying general boundary conditions.~I},
{Comm. Pure Appl. Math.}
\textbf{12}
{(1959)},
{no.~4},
{623--727}.

\bibitem{Agranovich97}
{M. S. Agranovich},
\textit{Elliptic boundary problems},
{Partial differential equations, IX},
{Encyclopaedia Math. Sci., Springer, Berlin},
{vol.~79},
{1997},
{pp.~1--144}.

\bibitem{AnopKasirenko16MFAT}
{A. V. Anop and T. M. Kasirenko},
\textit{Elliptic boundary-value problems in H\"ormander spaces},
{Methods Funct. Anal. Topology}
\textbf{22}
{(2016)},
{no.~4},
{295--310}.

\bibitem{AnopMurach14MFAT}
{A. V. Anop and A. A. Murach},
\textit{Parameter-elliptic problems and interpolation with a function parameter},
{Methods Funct. Anal. Topology}
\textbf{20}
{(2014)},
{no.~2},
{103--116}.

\bibitem{AnopMurach14UMJ}
{A. V. Anop and  A. A. Murach},
\textit{Regular elliptic boundary-value problems in the extended Sobolev scale},
{Ukrainian Math. J.}
\textbf{66}
{(2014)},
{no.~7},
{969--985}.

\bibitem{BerezanskyKreinRoitberg63}
{Ju.(Yu.) M. Berezanskii, S. G. Krein, and Ja.(Ya.) A. Roitberg}, \textit{A theorem on homeomorphisms and local increase of smoothness up to the boundary for solutions of elliptic equations} (Russian),
{Dokl. Akad. Nauk SSSR}
\textbf{148}
{(1963)},
{745--748}
{(English translation in: Dokl. Math. \textbf{4} (1963), 152--155)}.

\bibitem{Berezansky68}
{Yu. M. Berezansky},
\textit{Expansions in Eigenfunctions of Selfadjoint Operators},
{Transl. Math. Monogr., vol.~17},
{American Mathematical Society},
{Providence, RI},
{1968}.

\bibitem{BerghLefstrem80}
{J. Bergh and J. L\"ofstr\"om},
\textit{Interpolation Spaces},
{Grundlehren Math. Wiss., vol.~223},
{Springer},
{Berlin},
{1976}.

\bibitem{BinghamGoldieTeugels89}
{N. H. Bingham, C.~M. Goldie, and J.~L. Teugels},
\textit{Regular Variation},
{Encyclopedia Math. Appl., vol.~27},
{Cambridge University Press},
{Cambridge},
{1989}.

\bibitem{Bitsadze90}
{A. V. Bitsadze},
\textit{On the Neumann problem for harmonic functions},
{Dokl. Math.}
\textbf{41}
{(1990)},
{no. 2},
{193--195}.

\bibitem{Bonnaillie-NoelDambrineHerauVial10}
{V. Bonnaillie-Noel, M. Dambrine, F. Herau, and G. Vial},
\textit{On generalized Ventcel's type boundary conditions for Laplace operator in a bounded domain},
{SIAM J. Math. Anal.}
\textbf{42}
{(2010)},
{no. 2},
{931--945}.

\bibitem{ChepurukhinaMurach15MFAT}
{I. S. Chepurukhina and A. A. Murach},
\textit{Elliptic problems in the sense of B.~Lawruk on two-sided refined scales of spaces},
{Methods Funct. Anal. Topology}
\textbf{21}
{(2015)},
{no.~1},
{6--21}.

\bibitem{EidelmanZhitarashu98}
{S. D. Eidel'man and N. V. Zhitarashu},
\textit{Parabolic Boundary Value Problems},
{Oper. Theory Adv. Appl., vol.~101},
{Birkh\"auser},
{Basel},
{1998}.

\bibitem{FoiasLions61}
{C. Foia\c{s} and J.-L. Lions},
\textit{Sur certains th\'eor\`emes d'interpolation},
{Acta Scient. Math. Szeged}
\textbf{22}
{(1961)},
{no.~3--4},
{269--282}.

\bibitem{GohbergKrein60}
{I. Hohberg and M. Krein},
\textit{The basic propositions on defect numbers, root vectors, and indices of linear operators},
{Amer. Math. Soc. Transl. Ser. II}
\textbf{13}
{(1960)},
{no.~2},
{185--264}.

\bibitem{Hermander63}
{L. H\"ormander},
\textit{Linear Partial Differential Operators},
{Grundlehren Math. Wiss., vol.~116},
{Springer-Verlag},
{Berlin},
{1963}.

\bibitem{Hermander83}
{L. H\"ormander},
\textit{The Analysis of Linear Partial Differential Operators, vol.~II, Differential Operators with Constant Coefficients},
{Grundlehren Math. Wiss., vol.~257},
{Springer-Verlag},
{Berlin},
{1983}.

\bibitem{Hermander85}
{L. H\"ormander},
\textit{The Analysis of Linear Partial Differential Operators, vol.~III, Pseudo-Differential Operators},
{Grundlehren Math. Wiss., vol.~274},
{Springer-Verlag},
{Berlin},
{1985}.

\bibitem{Jacob010205}
{N. Jacob},
\textit{Pseudodifferential Operators and Markov Processes}
{(in 3 volumes)},
{Imperial College Press},
{London},
{2001, 2002, 2005}.

\bibitem{Karachik92}
{V. V. Karachik},
\textit{Solvability of a boundary value problem for the Helmholtz equation with higher-order normal derivatives on the boundary} {(Russian)},
{Differ. Uravn.}
\textbf{28}
{(1992)},
{no.~5},
{907--909}.

\bibitem{Karachik96}
{V. V. Karachik},
\textit{On a problem for the Poisson equation with higher-order normal derivatives on the boundary},
{Differ. Equ.}
\textbf{32}
{(1996)},
{no.~3},
{421--424}.

\bibitem{Karamata30a}
{J. Karamata},
\textit{Sur certains "Tauberian theorems" \;de M.~M.~Hardy et Litt\-lewood},
{Mathematica (Cluj)}
\textbf{3}
{(1930)},
{33--48}.

\bibitem{KostarchukRoitberg73}
{Ju.(Yu.) V. Kostarchuk and Ja.(Ya.) A. Roitberg},
\textit{Isomorphism theorems for elliptic boundary value problems with boundary conditions that are not normal}
{(Russian)},
{Ukra\"\i n. Mat. Zh.}
\textbf{25}
{(1973)},
{no.~2},
{277--283}
{(English translation in: Ukrainian Math. J. 25 (1973), no.~2, 222--226)}.

\bibitem{Kozhevnikov01}
{A. Kozhevnikov},
\textit{Complete scale of isomorphisms for elliptic pseudodifferential boundary-value problems},
{J. London Math. Soc.~(2)}
\textbf{64}
{(2001)},
{no.~2},
{409--422}.

\bibitem{KozlovMazyaRossmann97}
{V. A. Kozlov, V. G. Maz'ya, and J. Rossmann},
\textit{Elliptic Boundary Value Problems in Domains with Point Singularities},
{Math. Surveys Monogr., vol.~52},
{American Mathematical Society},
{Providence, RI},
{1997}.

\bibitem{Krasil'nikov61}
{V. N. Krasil'nikov},
\textit{On the solution of some boundary-contact problems of linear hydrodynamics},
{J. Appl. Math. Mech.}
\textbf{25}
{(1961)},
{1134--1141}.

\bibitem{KreinPetuninSemenov82}
{S.~G.~Krein, Yu.~L.~Petunin, and E.~M.~Sem\"enov}, \textit{Interpolation of Linear Operators},
{Transl. Math. Monogr., vol.~54},
{American Mathematical Society},
{Providence, R.I.},
{1982}.

\bibitem{Lawruk63a}
{B. Lawruk},
\textit{Parametric boundary-value problems for elliptic systems of linear differential equations. I.~Construction of conjugate problems}
{(Russian)},
{Bull. Acad. Polon. Sci. S\'{e}r. Sci. Math. Astronom. Phys.}
\textbf{11}
{(1963)},
{no.~5},
{257--267}.

\bibitem{LionsMagenes72}
{J.-L. Lions and E. Magenes},
\textit{Non-Homogeneous Boundary-Value Problems and Applications, vol.~I},
{Grundlehren Math. Wiss., vol.~181},
{Springer-Verlag},
{New York\,--\,Heidelberg},
{1972}.


\bibitem{Los17UMJ9}
{V. M. Los},
\textit{Classical solutions of parabolic initial-boundary-value problems and H\"ormander spaces},
{Ukrainian Math. J.}
\textbf{68}
{(2017)},
{no.~9},
{1412--1423}.

\bibitem{LosMikhailetsMurach17CPAA}
{V. Los, V. A. Mikhailets, and A. A. Murach},
\textit{An isomorphism theorem for parabolic problems in H\"ormander spaces and its applications},
{Commun. Pure Appl. Anal.}
\textbf{16}
{(2017)},
{no.~1},
{69--97}.

\bibitem{LosMurach13MFAT2}
{V. Los and A. A. Murach},
\textit{Parabolic problems and interpolation with a function parameter}, {Methods Funct. Anal. Topology}
\textbf{19}
{(2013)},
{no.~2},
{146--160}.

\bibitem{LosMurach17OpenMath}	
{V. Los and A. Murach},
\textit{Isomorphism theorems for some parabolic initial-boundary value problems in H\"ormander spaces},
{Open Math.}
\textbf{15}
{(2017)},
{57--76}.

\bibitem{LukyanovNazarov98}
{V. V. Luk'yanov and A. I. Nazarov},
\textit{Solution of the Venttsel' problem for the Laplace and the Helmholtz equations by means of iterated potentials},
{J. Math. Sci. (N. Y.)}
\textbf{102}
{(1998)},
{no. 4},
{4265--4274}.

\bibitem{LuoTrudinger91}
{Y. Luo and N. S. Trudinger},
\textit{Linear second order elliptic equations with Venttsel' boundary conditions},
{Proc. Roy. Soc. Edinburgh Sect. A}
\textbf{118}
{(1991)},
{no. 3--4},
{193--207}.

\bibitem{MikhailetsMurach05UMJ5}
{V. A. Mikhailets and A. A. Murach},
\textit{Elliptic operators in a refined scale of function spaces}, {Ukrainian. Math. J.}
\textbf{57},
{(2005)},
{no.~5},
{817--825}.

\bibitem{MikhailetsMurach06UMJ2}
{V. A. Mikhailets and A. A. Murach},
\textit{Refined scales of spaces, and elliptic boundary value problems.~I},
{Ukrainian Math. J.}
\textbf{58},
{(2006)},
{no.~2},
{244--262}.

\bibitem{MikhailetsMurach06UMJ3}
{V. A. Mikhailets and A. A. Murach},
\textit{Refined scales of spaces, and elliptic boundary value problems.~II},
{Ukrainian Math. J.}
\textbf{58}
{(2006)},
{no.~3},
{398--417}.

\bibitem{MikhailetsMurach06UMJ11}
{V. A. Mikhailets and A. A. Murach},
\textit{A regular elliptic boundary value problem for a homogeneous equation in a two-sided refined scale of spaces},
{Ukrainian Math.~J.}
\textbf{58}
{(2006)},
{no.~11},
{1748--1767}.

\bibitem{MikhailetsMurach06UMB4}
{V. A. Mikhailets and A. A. Murach},
\textit{An elliptic operator with homogeneous regular boundary conditions in a two-sided refined scale of spaces},
{Ukr. Math. Bull.}
\textbf{3}
{(2006)},
{no.~4},
{529--560}.

\bibitem{MikhailetsMurach07UMJ5}
{V. A. Mikhailets and A. A. Murach},
\textit{Refined scales of spaces, and elliptic boundary value problems.~III},
{Ukrainian Math. J.}
\textbf{59}
{(2007)},
{no.~5},
{744--765}.

\bibitem{MikhailetsMurach08UMJ4}
{V. A. Mikhailets and A. A. Murach},
\textit{An elliptic boundary-value problem in a two-sided refined scale of spaces},
{Ukrainian. Math. J.}
\textbf{60}
{(2008)},
{no.~4},
{574--597}.

\bibitem{09OperatorTheory191}
{V. A. Mikhailets and A. A. Murach},
\textit{Elliptic problems and H\"ormander spaces},
{Oper. Theory. Adv. Appl.}
\textbf{191}
{(2009)},
{447--470}.

\bibitem{MikhailetsMurach12BJMA2}
{V. A. Mikhailets and A. A. Murach},
\textit{The refined Sobolev scale, inter\-po\-la\-tion, and elliptic problems},
{Banach J. Math. Anal.}
\textbf{6}
{(2012)},
{no.~2},
{211--281}.

\bibitem{MikhailetsMurach13UMJ3}
{V. A. Mikhailets and A. A. Murach},
\textit{Extended Sobolev scale and elliptic operators},
{Ukrainian Math.~J.}
\textbf{65}
{(2013)},
{no.~3},
{435--447}.

\bibitem{MikhailetsMurach14}
{V. A. Mikhailets and A. A. Murach},
\textit{H\"ormander spaces, interpolation, and elliptic problems},
{De Gruyter Studies in Math., vol.~60},
{De Gruyter},
{Berlin},
{2014}.

\bibitem{MikhailetsMurach15ResMath1}
{V. A. Mikhailets and A. A. Murach}, \textit{Interpolation Hilbert spaces between Sobolev spaces},
{Results Math.}
\textbf{67}
{(2015)},
{no.~1},
{135--152}.

\bibitem{Murach08MFAT2}
{A. A. Murach},
\textit{Douglis-Nirenberg elliptic systems in the refined scale of spaces on a closed manifold},
{Methods Funct. Anal. Topology}
\textbf{14}
{(2008)},
{no.~2},
{142--158}.

\bibitem{MurachChepurukhina15UMJ}
{A. A. Murach and I. S. Chepurukhina},
\textit{Elliptic boundary-value problems in the sense of Lawruk on Sobolev and H\"ormander spaces},
{Ukrainian Math.~J.}
\textbf{67}
{(2015)},
{no.~5},
{764--784}.

\bibitem{MurachZinchenko13MFAT1}
{A. A. Murach and T. N. Zinchenko},
\textit{Parameter-elliptic operators on the extended Sobolev scale}, {Methods Funct. Anal. Topology}
\textbf{19}
{(2013)},
{no.~1},
{29--39}.

\bibitem{NicolaRodino10}
{F. Nicola and L. Rodino},
\textit{Global Pseudodifferential Calculus on Euclidean Spaces},
{Pseudo Diff. Oper., vol.~4},
{Birkh\"aser},
{Basel},
{2010}.

\bibitem{Paneah00}
{B. Paneah},
\textit{The Oblique Derivative Problem. The Poincar\'e Problem}, {Wiley--VCH},
{Berlin},
{2000}.

\bibitem{Peetre61}
{J. Peetre}, \textit{Another approach to elliptic boundary problems},  {Comm. Pure Appl. Math.}
\textbf{14}
{(1961)},
{no.~4},
{711--731}.

\bibitem{Peetre66}
{J. Peetre}, \textit{On interpolation functions},
{Acta Sci. Math. (Szeged)}
\textbf{27}
{(1966)},
{167--171}.

\bibitem{Peetre68}
{J. Peetre},
\textit{On interpolation functions.~II},
{Acta Sci. Math. (Szeged)}
\textbf{29}
{(1968)},
{91--92}.

\bibitem{Roitberg64}
{Ja.(Ya.) A. Roitberg},
\textit{Elliptic problems with non-homogeneous boundary conditions and local increase of smoothness of generalized solutions up to the boundary}
{(Russian)},
{Dokl. Akad. Nauk SSSR}
\textbf{157}
{(1964)},
{798--801}
{(English translation in: Dokl. Math. \textbf{5} (1964), 1034--1038)}.

\bibitem{Roitberg65}
{Ja.(Ya.) A. Roitberg},
\textit{A theorem on the homeomorphisms induced in $L_{p}$ by elliptic operators and the local smoothing of generalized solutions}
{(Russian)},
{Ukra\"\i n. Mat. Zh.}
\textbf{17}
{(1965)},
{no.~5},
{122--129}.

\bibitem{Roitberg68}
{Ja.(Ya.) A. Roitberg},
\textit{Theorems on homeomorphisms which can be realized by elliptic operators}
{(Russian)},
{Dokl. Akad. Nauk SSSR}
\textbf{180}
{(1968)},
{542--545}
{(English translation in: Dokl. Math. \textbf{9} (1968), 656--660)}.

\bibitem{Roitberg69}
{Ja.(Ya.) A. Roitberg},
\textit{Green’s formula and a theorem on homeomorphisms for general elliptic boundary value problems with boundary conditions which are not normal}
{(Russian)},
{Ukra\"\i n. Mat. Zh.}
\textbf{21}
{(1969)},
{no.~3},
{406--413}
{(English translation in: Ukrainian Math. J. \textbf{21} (1969), no.~3, 343--349)}.

\bibitem{Roitberg70}
{Ja.(Ya.) A. Roitberg},
\textit{Homeomorphism theorems and Green’s formula for general elliptic boundary value problems with boundary conditions that are not normal} {(Russian)},
{Mat. Sb.}
\textbf{83(125)}
{(1970)},
{no.~2(10)}
{181--213}
{(English translation in: Sb. Math. \textbf{12} (1970), no.~2, 177--212)}.

\bibitem{Roitberg96}
{Ya. Roitberg},
\textit{Elliptic Boundary Value Problems in the Spaces of Distributions},
{Math. Appl. (Springer), vol.~384},
{Kluwer Acad. Publ. Group},
{Dordrecht},
{1996}.

\bibitem{Roitberg99}
{Ya. Roitberg},
\textit{Boundary Value Problems in the Spaces of Distributions},
{Math. Appl. (Springer), Vol.~498},
{Kluwer Acad. Publ. Group},
{Dordrecht},
{1999}.

\bibitem{Schechter60}
{M. Schechter},
\textit{Negative norms and boundary problems},
{Ann. of Math. (2)}
\textbf{72}
{(1960)},
{no.~3},
{581--593}.

\bibitem{Schechter63a}
{M. Schechter},
\textit{On $L_{p}$ estimates and regularity,~I},
{Amer. J. Math.}
\textbf{85}
{(1963)},
{no.~1},
{1--13}.

\bibitem{Schechter63b}
{M. Schechter}, \textit{On $L_{p}$ estimates and regularity,~II},
{Math. Scand.}
\textbf{13}
{(1963)},
{no.~1},
{47--69}.

\bibitem{Schechter64}
{M. Schechter},
\textit{On $L_{p}$ estimates and regularity,~III},
{Ric. Mat.}
\textbf{13}
{(1964)},
{192--206}.

\bibitem{Seneta76}
{E. Seneta},
\textit{Regularly Varying Functions},
{Lecture Notes in Math., vol.~508},
{Springer-Verlag},
{Berlin},
{1976}.

\bibitem{Sokolovskiy88}
{V. B. Sokolovskii},
\textit{On a generalization of the Neumann problem}
{(Russian)},
{Differ. Uravn.}
\textbf{24}
{(1988)},
{no.~4},
{714--716}.

\bibitem{Stepanets05}
{A. I. Stepanets},
\textit{Methods of Approximation Theory},
{VSP},
{Utrecht},
{2005}.

\bibitem{Triebel95}
{H.~Triebel},
\textit{Interpolation Theory, Function Spaces, Differential Operators} {[2-nd edn]},
{Johann Ambrosius Barth},
{Heidelberg},
{1995}.

\bibitem{Triebel01}
{H. Triebel},
\textit{The Structure of Functions},
{Monogr. Math., vol.~97},
{Birkh\"auser},
{Basel},
{2001}.

\bibitem{VainbergGrushin67b}
{B. R. Vainberg and V. V. Grushin},
\textit{Uniformly nonelliptic problems,~II}
{(Russian)},
{Mat. Sb.}
\textbf{73(115)}
{(1967)},
{no.~1},
{126--154}
{(English translation in: Sb. Math. \textbf{2} (1967), 111--133)}.


\bibitem{Venttsel59}
{A. D. Ventcel},
\textit{On boundary conditions for multi-dimensional diffusion processes},
{Theory Probab. Appl.}
\textbf{4}
{(1959)},
{164--177}.

\bibitem{VeshevKouzov77}
{V. A. Veshev and D. P. Kouzov},
\textit{Influence of the medium on the vibrations of plates joined at right angles},
Acoustical Physics
\textbf{23}
{(1977)},
{no.~3},
{206--211}.

\bibitem{VolevichPaneah65}
{L. R. Volevich and B. P. Paneah},
\textit{Certain spaces of generalized functions and embedding theorems} {(Russian)},
{Uspehi Mat. Nauk}
\textbf{20}
{(1965)},
{no.~1},
{3--74}
{(English translation in: Russian Math. Surveys \textbf{20} (1965), no.~1, 1--73)}.

\bibitem{Zinchenko17OpenMath}
{T. Zinchenko}; \textit{Elliptic operators on refined Sobolev scales on vector bundles},
{Open Math.}
\textbf{15}
{(2017)},
{907--925}.

\bibitem{ZinchenkoMurach12UMJ11}
{T. N. Zinchenko and A. A. Murach},
\textit{Douglis--Nirenberg elliptic systems in H\"ormander spaces}, {Ukrainian Math.~J.}
\textbf{64}
{(2013)},
{no.~11},
{1672--1687}.

\end{thebibliography}
\end{document}